\documentclass[12pt]{amsart}

\usepackage[margin=1in]{geometry}

\usepackage{amsthm,amsmath,amsfonts,amssymb}
\usepackage{ytableau}
\ytableausetup{smalltableaux}
\usepackage{enumerate}

\usepackage{graphicx}

\usepackage{tikz, tikz-3dplot, pgfplots}
\usepackage{tkz-graph}
\usetikzlibrary{matrix}
\usetikzlibrary[positioning,patterns] % tikz libraries for relative positioning and silly patterns
\usepackage{color}
\usepackage{shuffle}

\newtheorem{thm}{Theorem}[section]
\newtheorem{prop}[thm]{Proposition}
\newtheorem{cor}[thm]{Corollary}
\newtheorem{lemma}[thm]{Lemma}
\theoremstyle{definition}
\newtheorem{example}[thm]{Example}
\newtheorem{defn}[thm]{Definition}
\newtheorem{quest}[thm]{Question}
\theoremstyle{remark}
\newtheorem{remark}[thm]{Remark}

\newcommand{\Z}{\mathbb{Z}}
\newcommand{\x}{\mathbf{x}}
\newcommand{\C}{\mathbb{C}}
\newcommand{\QSym}{\mathrm{QSym}}
\newcommand{\KP}{K_P(\x)}
\newcommand{\KPomega}{K_{(P, \omega)}(\x)}
\newcommand{\KPmin}{\widetilde{K}_P(\x)}
\newcommand{\KQ}{K_Q(\x)}

\newcommand{\des}{\mathrm{des}}
\newcommand{\Pomega}{(P, \omega)}

\newcommand{\lm}{\lambda / \mu}
\newcommand{\vc}[1]{\rule[-3ex]{0pt}{0pt}\rule[4ex]{0pt}{0pt} \raisebox{\dimexpr-.5\height+.5\ht\strutbox\relax}{#1}}

\DeclareMathOperator{\rev}{rev}
\DeclareMathOperator{\co}{co}
\DeclareMathOperator{\sh}{sh}
\DeclareMathOperator{\wt}{wt}
\DeclareMathOperator{\height}{ht}

\DeclareMathOperator{\type}{type}

\DeclareMathOperator{\amb}{amb}

\DeclareMathOperator{\sign}{sign}
\DeclareMathOperator{\comp}{wt}

\newcommand{\V}{V}
\DeclareMathOperator{\MinOne}{\eta}
\DeclareMathOperator{\MaxOne}{\tilde{\eta}}

\title{$P$-Partitions and Quasisymmetric Power Sums}

\author{Ricky Ini Liu}
\address{Department of Mathematics, North Carolina State University, Raleigh, NC}
\email{riliu@ncsu.edu}
\author{Michael Weselcouch}
\address{Department of Mathematics, North Carolina State University, Raleigh, NC}
\email{mweselc@ncsu.edu}

\thanks{R. I. Liu and M. Weselcouch were partially supported by National Science Foundation grant DMS-1700302.}

\date{\today}

\begin{document}
\maketitle

\begin{abstract}
The \emph{$\Pomega$-partition generating function} of a labeled poset $\Pomega$ is a quasisymmetric function enumerating certain order-preserving maps from $P$ to $\Z^+$.  We study the expansion of this generating function in the recently introduced type 1 quasisymmetric power sum basis $\{\psi_\alpha\}$.  Using this expansion, we show that connected, naturally labeled posets have irreducible $P$-partition generating functions.  We also show that series-parallel posets are uniquely determined by their partition generating functions.  We conclude by giving a combinatorial interpretation for the coefficients of the $\psi_\alpha$-expansion of the $\Pomega$-partition generating function akin to the Murnaghan-Nakayama rule.
\end{abstract}

\section{Introduction}
For a finite poset $P=(P, \prec)$ with a labeling $\omega\colon P \to [n]$, the \emph{$\Pomega$-partition generating function} $\KPomega$ is a quasisymmetric function enumerating certain order-preserving maps from $P$ to $\Z^+$. In this paper, we study the expansion of $\KPomega$ in the type 1 quasisymmetric power sum basis  $\{\psi_\alpha\}$ introduced by Ballantine, Daugherty, Hicks, Mason, and Niese in \cite{BallantineEtc}. Such expansions have previously been considered in the naturally labeled case by Alexandersson and Sulzgruber in \cite{AlexanderssonSulzgruber}.

It is well known that the $\Pomega$-partition generating function of a disconnected poset is reducible since it can be expressed as a product over the connected components of $P$. McNamara and Ward \cite{McNamaraWard} asked whether the converse is true, namely whether the $\Pomega$-partition generating function of a connected poset is always irreducible in $\QSym$. We answer this question affirmatively in the naturally labeled case. In particular, we relate part of the $\psi_\alpha$-expansion of $\KP$ to certain \emph{zigzag labelings} of $P$, which exist only for connected posets. We then use this to deduce that connected, naturally labeled posets have irreducible $P$-partition generating functions in Section 4. It remains open whether this is true for all connected labeled posets $(P, \omega)$. (See \cite{LamPylyavskyy} for some discussion of irreducibility in $\QSym$.)

An important object in this study is a pair of linear functionals $\MinOne$ and $\MaxOne$ on $\QSym$ (defined by the present authors in \cite{LiuWeselcouch}) that can be used to determine if a poset has exactly $1$ minimal or $1$ maximal element. These functionals also have the property that they send any reducible element of $\QSym$ to $0$. In Section $3$, we express $\MinOne$ and $\MaxOne$ in terms of the basis $\{\psi_\alpha\}$ and use this to describe the action of various involutions of $\QSym$ on this basis.

In \cite{LiuWeselcouch}, the present authors studied the question of when two distinct (naturally labeled) posets can have the same $P$-partition generating function.  This question has been studied extensively in the case of skew Schur functions \cite{BTvW, MvW, RSvW}, by McNamara and Ward \cite{McNamaraWard} for general labeled posets, and by Hasebe and Tsujie \cite{HasebeTsujie} for rooted trees (and more generally $(N, \bowtie)$-free posets). Hasebe and Tsujie ask in %Problem 6.1 in
\cite{HasebeTsujie} whether series-parallel posets can be distinguished by their partition generating functions. %They ask this as a way to generalize their result that  can be distinguished by their partition generating functions. 
(A poset is \emph{series-parallel} if it %is $N$-avoiding, meaning that there does not exist an induced subposet $\{a, b, c, d\}$ with exactly the relations $a \prec b \succ c \prec d$.)
can be built from one-element posets using ordinal sum and disjoint union operations, or alternatively, if it is $N$-free.)
We use the previously stated result on irreducibility to give a complete, affirmative answer to this question in Section 5.

Alexandersson and Sulzgruber \cite{AlexanderssonSulzgruber} show that when $\Pomega$ is naturally labeled, $\KP = \sum_\alpha c_{\alpha}\psi_\alpha$ is $\psi$-positive, and they give a combinatorial interpretation for the coefficients $c_{\alpha}$.  In Section 6, we extend this work and give a (signed) combinatorial interpretation for the coefficients in the $\psi_{\alpha}$-expansion of $\KPomega$ for any labeled poset. This interpretation generalizes the Murnaghan-Nakayama rule for computing the expansion of a skew Schur function in terms of power sum symmetric functions.

In summary, in Section 2 we will give some preliminary information; in Section 3 we describe $\MinOne$, $\MaxOne$, and various involutions of $\QSym$ in terms of the type 1 quasisymmetric power sum basis $\{\psi_\alpha\}$; in Section 4 we show that connected, naturally labeled posets have irreducible $P$-partition generating functions; in Section 5 we show that series-parallel posets are uniquely determined by their partition generating function; and in Section 6 we give a combinatorial interpretation for the coefficients in the $\psi_{\alpha}$-expansion of $\KPomega$ akin to the Murnaghan-Nakayama rule.

%We say that a poset $P$ is \emph{naturally labeled} if $x \preceq y$ implies $x \leq y$ as integers.  Alexandersson and Sulzgruber in \cite{AlexanderssonSulzgruber} showed that when $\Pomega$ is naturally labeled $\KP = \sum_\alpha c_{\alpha}\psi_\alpha$ is $\psi$-positive and they gave a combinatorial description of the coefficient $c_{\alpha}$ of $\psi_{\alpha}$.  

%In this paper we will show that for all connected, naturally labeled posets $P$, the $P$-partition generating function is irreducible over $\QSym$.  This gives a partial answer to Question $7.2$ asked by McNamara and Ward in \cite{McNamaraWard}.  In order to show this result, we need to expand $\KP$ in terms of the type 1 power sum quasisymmetric function basis $\{\psi_\alpha\}$ introduced by Ballantine, Daugherty, Hicks, Mason, and Niese in \cite{BallantineEtc}.  Alexandersson and Sulzgruber \cite{AlexanderssonSulzgruber} showed that $\KP = \sum_\alpha c_{\alpha}\psi_\alpha$ is $\psi$-positive and they gave a combinatorial description of the coefficient $c_{\alpha}$ of $\psi_{\alpha}$.  

\section{Preliminaries}
We begin with some preliminaries about compositions, posets, quasisymmetric functions, and Hopf algebras.
%posets, compositions, and quasisymmetric functions.
For more information, see \cite{GrinbergReiner, LiuWeselcouch, McNamaraWard, Stanley2}. 

\subsection{Compositions and partitions}
A \textit{composition} $\alpha= (\alpha_1, \alpha_2, \dots, \alpha_k)$ of $n$ is a finite sequence of positive integers summing to $n$.  (When it is clear from context, we will remove the parentheses and commas when writing a composition.) %We denote the set of all compositions of $n$ by $\Comp_n$.
A \emph{weak composition} of $n$ is a finite sequence of nonnegative integers summing to $n$.  A \emph{partition} of $n$ is a composition of $n$ whose parts are in weakly decreasing order.  Given a composition $\alpha$ and partition $\lambda$, we write $\alpha \sim \lambda$ if $\lambda$ is formed by rearranging the parts of $\alpha$ into weakly decreasing order.  We use the notation $\alpha \vDash n$ if $\alpha$ is a composition of $n$ and $\lambda \vdash n$ if $\lambda$ is a partition of $n$.

We will use the shorthand $1^k$ to denote the composition $(\underbrace{1, 1, \dots, 1}_k)$.
The \emph{reverse of $\alpha$}, denoted $\alpha^{\textrm{rev}}$, is the composition formed by reversing the order of $\alpha$.  The \emph{length of $\alpha$}, denoted $l(\alpha)$, is the number of parts of $\alpha$.

The compositions of $n$ are in bijection with the subsets of $[n-1]$ in the following way: for any composition $\alpha$, define \[D(\alpha) = \{\alpha_1, \quad \alpha_1 + \alpha_2, \quad \dots, \quad \alpha_1 + \alpha_2 +\dots +\alpha_{k-1}\} \subseteq [n-1].\]
Likewise, for any subset $S = \{s_1, s_2, \dots, s_{k-1}\}\subseteq [n-1]$ with $s_1<s_2<\dots < s_{k-1}$, we can define the composition
\[\co(S) = (s_1, \quad s_2-s_1, \quad s_3-s_2, \quad \dots, \quad s_{k-1}-s_{k-2}, \quad n-s_{k-1}).\]

Given two nonempty compositions $\alpha = (\alpha_1, \alpha_2, \dots, \alpha_k)$ and $\beta = (\beta_1, \beta_2, \dots, \beta_l)$, their \emph{concatenation} is
\[\alpha \cdot \beta = (\alpha_1, \alpha_2, \dots, \alpha_k, \beta_1, \beta_2, \dots, \beta_l),\]
and their \emph{near-concatenation} is
\[ \alpha \odot \beta = (\alpha_1, \alpha_2, \dots, \alpha_k + \beta_1, \beta_2, \dots, \beta_m). \]
Observe that if $\alpha \vDash n$ and $\beta \vDash m$, then both $\alpha \cdot \beta \vDash (n+m)$ and $\alpha \odot \beta \vDash (n+m)$.

For a partition $\lambda$, define $z_\lambda = 1^{m_1}m_1! \cdot 2^{m_2}m_2!\cdots $, with $m_i$ being the multiplicity of $i$ in $\lambda$.  This number is the size of the centralizer of a group element $g \in S_n$ whose cycle type is $\lambda$ \cite{Stanley2}.  For a composition $\alpha$ where $\alpha \sim \lambda$, we define $z_\alpha =z_\lambda$.% \cite{SaganBook}.

%For a composition $\alpha$, define $z_\alpha = 1^{m_1}m_1! \cdot 2^{m_2}m_2!\cdots $, with $m_i$ being the multiplicity of $i$ in $\alpha$.  When $\alpha$ is a partition, this number is the size of the centralizer of a group element $g \in S_n$ whose cycle type is $\alpha$ \cite{Stanley2}. %\cite{SaganBook}.

\subsubsection{Refinement}

If $\alpha$ and $\beta$ are both compositions of $n$, then we say that $\alpha$ \emph{refines} $\beta$ (equivalently, $\beta$ \emph{coarsens} $\alpha$), denoted $\alpha \preceq \beta$, if
\[ \beta = (\alpha_1 + \cdots + \alpha_{i_1}, \quad \alpha_{i_1+1} + \cdots + \alpha_{i_1+i_2}, \quad \dots,\quad  \alpha_{i_1+ \cdots + i_{k-1} + 1} + \cdots + \alpha_{i_1+\cdots+i_k}),\]
for some $i_1, i_2, \dots, i_k$ summing to $l(\alpha)$. Equivalently, $\alpha \preceq \beta$ if and only if $D(\beta) \subseteq D(\alpha)$.

The following function will be important in a number of combinatorial formulas.
\begin{defn}
Given a refinement $\alpha$ of $\beta$, let $\alpha^{(i)}$ be the composition consisting of the parts of $\alpha$ that combine to form $\beta_i$, so $\alpha^{(i)} \vDash \beta_i$.  We then define
\[ \pi(\alpha) = \prod_{i=1}^{l(\alpha)}\sum_{j=1}^i\alpha_j \text{\hspace{.5 cm} and \hspace{.5 cm}} \pi(\alpha, \beta) = \prod_{i=1}^{l(\beta)}\pi(\alpha^{(i)}). \]
\end{defn}

%\begin{lemma}
%Let $\alpha$ be a composition of $n$ and let $D(\alpha) = \{s_1, s_2, \dots, s_{l-1}\}$ with $s_1 < s_2 < \dots < s_{l-1}<s_l=n$.  Then $\frac{n!}{\pi(\alpha)}$ counts the total number of permutation $\sigma \in S_n$ such that for $i = 1, \dots, l$, $\max\{\sigma_1, \dots, \sigma_{i}\} = \sigma_{i}$.
%\end{lemma}
Observe that $\pi(\alpha) = \pi(\alpha, (n))$.

\begin{example}
If $\alpha = (1,1,4,2,1)$ and $\beta = (2,7)$, then $\alpha \preceq \beta$, and 
\[\pi(\alpha, \beta) = (1\cdot(1+1))(4\cdot(4+2)\cdot(4+2+1))=336.\]
\end{example}

This function has the following combinatorial interpretation. Given compositions $\alpha, \beta \vDash n$ with $\alpha \preceq \beta$, let
\begin{align*}
D(\alpha) &= \{s_1, s_2, \dots, s_{l-1}\},& 0 &=s_0 <s_1 < s_2 < \dots < s_{l-1}<s_l=n,\\
D(\beta) &= \{s_{i_1}, s_{i_2}, \dots, s_{i_{k-1}}\}, & 0 &= i_0 < i_1 < i_2 < \cdots < i_{k-1} < i_k = l.
\end{align*}
Let $\Pi(\alpha,\beta)$ (denoted $\textrm{Cons}_{\alpha \preceq \beta}$ in \cite{BallantineEtc}) be the set of permutations $\sigma \in S_n$ such that  $\sigma_j \leq \sigma_{s_i}$ for all $s_{i_{m-1}} < j \leq s_i \leq s_{i_{m}}$, where $1 \leq m \leq k$. Alternatively, these are the permutations that satisfy $\sigma_{s_i} = \max_{j \in (s_{i_{m-1}},  s_i]} \sigma_j$ whenever $i_{m-1} < i \leq i_m$. In other words, suppose that the letters in $\sigma$ are broken into blocks with sizes given by the parts of $\beta$ and then into subblocks with sizes given by the parts of $\alpha$. Then the maximum of any subblock must occur at the end of that subblock, and these maximum entries must increase within any block.

\begin{example}
Let $\alpha = (1,1,4,2,1)$ and $\beta = (2,7)$. It follows that $D(\alpha) = \{1, 2, 6, 8 \}$ and $D(\beta) = \{2 \}$.  Then (using bars to indicate the composition $\beta$ and spaces to indicate the refinement given by $\alpha$)
\[\sigma = 2\,3\,|\,6417\,58\,9 \in \Pi(\alpha, \beta),\]
but 
\[\sigma' = 1\,4\,|\,2738\,56\,9 \notin \Pi(\alpha, \beta)\]
because $\sigma_8 \neq \max\{\sigma_3, \dots, \sigma_8\}$.
\end{example}

The following result is Lemma 3.7 from \cite{BallantineEtc}; we sketch a proof here for completeness.

\begin{lemma}\label{counting pi alpha beta}
For compositions $\alpha \preceq \beta$,
\[|\Pi(\alpha,\beta)| = \frac{n!}{\pi(\alpha,\beta)}.\]
%Given a refinement $\alpha$ of $\beta$, let $l = l(\alpha)$ and $k = l(\beta)$.  Let $D(\alpha) = \{s_1, s_2, \dots, s_{l-1}\}$ with $0 =s_0 <s_1 < s_2 < \dots < s_{l-1}<s_l=n$.  Let $D(\beta) = \{d_1, d_2, \dots, d_{k-1}\}$ with $0 = d_0< d_1 < d_2 < \dots < d_{k-1}< d_k = n$. Then $\frac{n!}{\pi(\alpha, \beta)}$ counts the number of permutation $\sigma \in S_n$ 
%such that for all $i$ and $j$ satisfying $d_{j-1} < s_i \leq d_{j}$, we have $\max\{\sigma_{d_{j-1}+1}, \dots, \sigma_{s_i}\} = \sigma_{s_i}$.
%such that for $j = 1, \dots, k$ and $i$ such that $d_{j-1} < s_i \leq d_{j}$, $\max\{\sigma_{d_{j-1}+1}, \dots, \sigma_{s_i}\} = \sigma_{s_i}$.
\end{lemma}

\begin{proof}
Choose $i_{m-1} < i \leq i_{m}$. For a random permutation $\sigma \in S_n$, the probability that $\sigma_{s_i} = \max_{j \in (s_{i_{m-1}},  s_i]} \sigma_j$ is \[\frac{1}{s_i - s_{i_{m-1}}} = \frac{1}{\alpha^{(m)}_1 + \cdots + \alpha^{(m)}_{i-i_{m-1}}}.\]
It is easy to check that these probabilities are independent, and taking the product over all $i$ gives $\frac{1}{\pi(\alpha,\beta)}$.
%This is Lemma 3.7 from \cite{BallantineEtc}. %Ballantine, Daugherty, Hicks, Mason, and Niese paper.
\end{proof}

%Define $\Pi(\alpha, \beta)$ to be the set of permutation described above (this set is called ``$\textrm{Cons}_{\alpha \preceq \beta}$" in \cite{BallantineEtc}). %Ballantine, Daugherty, Hicks, Mason, and Niese paper..

%\color{red}\textbf{An example here would be nice.  Could be similar to example 1 on page 7 of Ballantine etc. paper.}\color{black}

Using Lemma~\ref{counting pi alpha beta}, one can prove the following identity, which we will need later.

\begin{lemma}\label{pi sum}  Let $\alpha$ be a composition of $n$.  Then
\[ \frac{n!}{\pi(\alpha^{\textrm{rev}})} = \sum_{\beta \succeq \alpha} (-1)^{l(\alpha)-l(\beta)} \frac{n!}{\pi(\alpha, \beta)} .\]
\end{lemma}

\begin{proof}
Let $l(\alpha) = l$ and $D(\alpha) = \{s_1, s_2, \dots, s_{l-1}\}$ with $0 = s_0< s_1 < s_2 < \dots < s_{l-1}< s_l = n$.

By an argument similar to the one used in the proof of Lemma \ref{counting pi alpha beta}, the left hand side counts the number of permutations $\sigma$ such that for all $i = 1, \dots, l$,  $\max\{\sigma_{s_{i-1}+1}, \dots, \sigma_n\} = \sigma_{s_i}$.  In other words, out of the last $\alpha_i + \dots +\alpha_l$ values of $\sigma$, the largest of those values is in position $s_i$.  Observe that these are the permutations in $\Pi(\alpha, \alpha)$ such that $\sigma_{s_1} > \sigma_{s_2} > \dots > \sigma_{s_l}$.

The right hand side counts pairs $(\sigma, \beta)$ such that $\sigma \in \Pi(\alpha, \beta) \subseteq \Pi(\alpha, \alpha)$ with a sign depending on the length of $\beta$.  We will describe a sign-reversing involution whose fixed points are the permutations that are counted by the left hand side.
Let $j$ be the smallest positive number such that $\sigma_{s_{j}} < \sigma_{s_{j+1}}$. Then $\sigma \in \Pi(\alpha, \beta')$, where
\[D(\beta') = \begin{cases}
D(\beta) \cup \{s_j\} & \text{if $s_j \notin D(\beta)$,}\\
D(\beta) \setminus \{s_j\} & \text{if $s_j \in D(\beta)$.}
\end{cases}\]
%If $s_{j} \notin D(\beta)$, then $\sigma \in \Pi(\alpha, \beta')$ where $\beta' = \co(D(\beta) \cup \{s_j\})$. % is the composition that satisfies $D(\beta') = D(\beta) \cup \{s_{j}\}$.  %$D(\beta') = D(\beta) \cup \{s_j, n\}$.
%In this case, $l(\beta') = l(\beta) +1$.
%
%If $s_{j} \in D(\beta)$, then $\sigma \in \Pi(\alpha, \beta')$ where $\beta' = \co(D(\beta) \setminus \{s_j\})$.
In either case, $l(\beta') = l(\beta) \pm 1$, so $(\sigma, \beta) \mapsto (\sigma, \beta')$ is a sign-reversing involution. The fixed points are exactly the permutations in $\Pi(\alpha, \alpha)$ where $\sigma_{s_1} > \sigma_{s_2} > \dots > \sigma_{s_l}$, which are counted by the left hand side.
 % is the composition that satisfies $D(\beta') = D(\beta) \setminus \{s_{j}\}$. In this case, $l(\beta') = l(\beta) -1$.  Since we picked $j$ to be the smallest such number, this is an involution.  The fixed points are exactly the permutations where $\sigma_{s_1} > \sigma_{s_2} > \dots > \sigma_{s_l}$.  These are counted by the left hand side.
%The right hand side counts the number of permutations $\sigma \in Cons_{\alpha \preceq \beta}$ with a sign.  \textbf{Sign reversing involution} whose fixed points are the left hand side.  Here is a sketch of the sign reversing involution.  Look for $\alpha$-block with the largest entry $\alpha(i)$ such that the largest entry in $\alpha(i)$ is greater than the largest entry in $\alpha(i-1)$.  The block $\alpha(i)$ must be the rightmost block in its $\beta$-block.  If $\alpha(i)$ and $\alpha(i-1)$ are not in the same $\beta$-block, add $\alpha(i)$ to the $\beta$-block containing $\alpha(i-1)$.  If they are already in a $\beta$-block, remove $\alpha(i)$ from the $\beta$-block.  The fixed points are the permutations counted on the left hand side.  
\end{proof}

\subsubsection{Shuffles}

An important notion when working with compositions and quasisymmetric functions is that of shuffles.

\begin{defn}
Let $\alpha = (\alpha_1, \alpha_2, \dots, \alpha_k)$ and $\beta = (\beta_1, \beta_2, \dots, \beta_l)$. %For each subset $S \subseteq [k+l]$ of size $k$, define $\gamma_S$ to be the word whose subword given by positions in $S$ is $\alpha$ and whose subword given by positions in $[k+l]\setminus S$ is given by $\beta$. 
The \emph{multiset of shuffles} $\alpha \shuffle \beta$ consists of all $\binom{k+l}{k}$ compositions of length $k+l$ that contain $\alpha$ and $\beta$ as disjoint subsequences, counted with multiplicity.
% is defined by 
%\[ \alpha \shuffle \beta = \{(\gamma_{\sigma_1}, \gamma_{\sigma_2}, \dots, \gamma_{\sigma_{k+l}}) : \sigma \in \Sh_{k, l} \} \]
%where $\gamma = (\gamma_1, \gamma_2, \dots, \gamma_{k+l})$ is the concatenation $\alpha \cdot \beta$, and $\Sh_{k,l}$ is the subset of permutations
%\[\Sh_{k,l}=\{\sigma \in S_{k+l} : \sigma^{-1}_1 < \sigma^{-1}_2 < \dots < \sigma^{-1}_k ;\; \sigma^{-1}_{k+1} < \sigma^{-1}_{k+2} < \dots < \sigma^{-1}_{k+l}\}.\]
\end{defn}

\begin{example}
Let $\alpha = (4, 2, 1)$ and $\beta = (3, 1)$, then the multiset of shuffles of $\alpha$ and $\beta$ is:
\[ (4, 2, 1, 3, 1), (4, 2, 3, 1, 1), (4, 2, 3, 1, 1), (4, 3, 2, 1, 1), (4, 3, 2, 1, 1), \]
\[(4, 3, 1, 2, 1), (3, 4, 2, 1, 1), (3, 4, 2, 1, 1), (3, 4, 1, 2, 1), (3, 1, 4, 2, 1). \]
\end{example}

The shuffle operator $\shuffle$ is commutative, meaning $\alpha \shuffle \beta = \beta \shuffle \alpha$ as multisets.

\subsection{Posets and $P$-partitions}
Let $P = (P, \prec)$ be a finite poset of size $n$. A \emph{labeling} of $P$ is a bijection $\omega \colon P \to \{1, 2, \dots, n\}$.
\begin{defn} For a labeled poset $(P, \omega)$, a \emph{$(P, \omega)$-partition} is a map $\theta \colon P \to \Z^+$ that satisfies the following:
\begin{enumerate}[(a)]
\item If $x \preceq y$, then $\theta(x) \leq \theta(y)$.
\item If $x \preceq y$ and $\omega(x) > \omega(y)$, then $\theta(x) < \theta(y)$.
\end{enumerate}
\end{defn}

%We say that the \emph{weight} of a $\Pomega$-partition $\theta$ is the weak composition \[\wt(\theta) = (|\theta^{-1}(1)|, |\theta^{-1}(2)|, \dots ).\]

\begin{defn} The \emph{$(P, \omega)$-partition generating function} $K_{(P, \omega)}(x_1, x_2, \dots)$ for a labeled poset $(P, \omega)$ is given by 
\[K_{(P, \omega)}(\mathbf x) = K_{(P, \omega)}(x_1, x_2, \dots ) = \sum_{(P, \omega)\text{-partition } \theta}x_1^{|\theta^{-1}(1)|}x_2^{|\theta^{-1}(2)|}\cdots,\]
where the sum ranges over all $(P, \omega)$-partitions $\theta$.
\end{defn}
Note that $K_{(P, \omega)}(x_1, x_2, \dots )$ depends only on the relative order of $\omega(x)$ and $\omega(y)$ when $y$ covers $x$.  In the Hasse diagram of $\Pomega$, we will use a bold edge (or \emph{strict edge}) to represent when $x \prec y$ but $\omega(x) > \omega(y)$, while we will use a plain edge (or \emph{natural edge}) when $x \prec y$ and $\omega(x) < \omega(y)$.  We will replace the nodes of a labeled poset by circles with the label inside. %Note that we may sometimes treat a labeled poset $(P, \omega)$ as the equivalent poset $P$ on the ground set $[n]$.

\begin{example}
%The following labeled poset has $\omega(x) > \omega(y) > \omega(z)$.
%\begin{center}
%\begin{tikzpicture}
% [auto,
% vertex/.style={circle,draw=black!100,fill=black!100,thick,inner sep=0pt,minimum size=1mm}]
%\node (v1) at ( 0, 0) [vertex, label=below:$x$] {};
%%\node (v1) [draw,circle,minimum size=.45cm,inner sep=0pt] at (0,0) {$x$};
%\node (v2) at ( 1,1) [vertex, label=above:$y$] {};
%%\node (v2) [draw,circle,minimum size=.45cm,inner sep=0pt] at (1,1) {$y$};
%\node (v3) at (2,0)  [vertex, label=below:$z$] {};
%%\node (v3) [draw,circle,minimum size=.45cm,inner sep=0pt] at (2,0) {$z$};
%\draw [very thick] [-] (v1) to (v2);
%\draw [-] (v2) to (v3);
%\end{tikzpicture}
%\end{center}
%Hence we must have $\omega(x) = 3$, $\omega(y) = 2$, and $\omega(z)=1$, so we may draw $\Pomega$ as follows.
The following is a depiction of a labeled poset $(P, \omega)$ for $n=3$. Note that there is one strict edge and one natural edge.
 
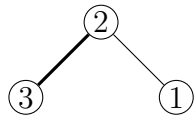
\begin{center}
\begin{tikzpicture}
 [auto,
 vertex/.style={circle,draw=black!100,fill=black!100,thick,inner sep=0pt,minimum size=1mm}]
\node (v1) [draw,circle,minimum size=.45cm,inner sep=0pt] at (0,0) {$3$};
\node (v2) [draw,circle,minimum size=.45cm,inner sep=0pt] at (1,1) {$2$};
\node (v3) [draw,circle,minimum size=.45cm,inner sep=0pt] at (2,0) {$1$};
\draw [very thick] [-] (v1) to (v2);
\draw [-] (v2) to (v3);
\end{tikzpicture}
\end{center} 
\end{example}
If $\omega$ is order-preserving, then $K_{(P, \omega)}(x_1, x_2, \dots )$ depends only on the structure of $P$. In this case, we call $P$ \emph{naturally labeled} and write $K_{P}(\mathbf x) = K_{(P, \omega)}(\mathbf x)$.

\begin{remark} \label{Disconnect posets}
If $P$ is a disconnected poset, then $K_{(P,\omega)}(\mathbf x)$ can be written as a product of $K_{(P_i, \omega_i)} (\mathbf x)$, where the $P_i$ are the connected components of $P$ (and $\omega_i$ is obtained by appropriately restricting $\omega$ to $P_i$).
\end{remark}

%In this paper, we will usually restrict our attention to the case when $P$ is $\emph{naturally labeled}$, that is, when $\omega$ is an order-preserving map. In this case, $\KP$ does not depend on our choice of natural labeling but only on the underlying structure of $P$.

A \emph{linear extension} of a labeled poset $(P, \omega)$ %(corresponding to a poset $P$ with ground set $[n]$)
is a permutation $\pi$ of $[n]$ that respects the relations in $(P, \omega)$, that is, if $x \preceq y$, then $\pi^{-1}(\omega(x)) \leq \pi^{-1}(\omega(y))$.  The set of all linear extensions of $(P, \omega)$ is denoted $\mathcal{L}\Pomega$. Note that $|\mathcal L\Pomega|$ is the coefficient of $x_1x_2\cdots x_n$ in $\KPomega$.

\subsection{Quasisymmetric Functions}
A \emph{quasisymmetric function} in the variables $x_1, x_2, \dots$ (with coefficients in $\C$) is a formal power series $f(\mathbf{x}) \in \C[[\mathbf{x}]]$ of bounded degree such that, for any composition $\alpha$, the coefficient of $x_1^{\alpha_1}x_2^{\alpha_2}\cdots x_k^{\alpha_k}$ equals the coefficient of $x_{i_1}^{\alpha_1}x_{i_2}^{\alpha_2}\cdots x_{i_k}^{\alpha_k}$ whenever $i_1 < i_2 < \dots < i_k$.  We denote the algebra of quasisymmetric functions by $\QSym = \bigoplus_{n \geq 0} \QSym_n$, graded by degree.

We will consider the following bases for $\QSym$: the monomial basis $\{M_\alpha\}$, the fundamental quasisymmetric function basis $\{L_\alpha\}$, and the (unnormalized) type 1 quasisymmetric power sum basis $\{\psi_{\alpha}\}$ introduced in \cite{BallantineEtc}. 

The \textit{monomial quasisymmetric function basis} $\{M_\alpha\}$, indexed by compositions $\alpha$, is given by 
\[M_\alpha = \sum_{1\leq i_1<i_2<\dots <i_k}  x_{i_1}^{\alpha_1}x_{i_2}^{\alpha_2}\cdots x_{i_k}^{\alpha_{k}}.\]
%For example, $M_{(2, 1)} = \sum_{i<j}x_i^2x_j$.

The \textit{fundamental quasisymmetric function basis} $\{L_\alpha\}$, also indexed by compositions $\alpha$, is given by
\[L_\alpha = \sum_{\substack{i_1 \leq \dots \leq i_n\\i_s < i_{s+1} \text{ if } s \in D(\alpha)}} x_{i_1}x_{i_2} \cdots x_{i_n}.\]
In terms of the monomial basis,
%\[L_\alpha = \sum_{\beta \preceq \alpha}M_\beta,\]
$L_\alpha = \sum_{\beta \preceq \alpha}M_\beta$,
where the sum runs over all refinements $\beta$ of $\alpha$.  %(A composition $\beta$ is a \emph{refinement} of $\alpha$ if $\alpha$ can be obtained by adding together adjacent parts of the composition $\beta$.)
By M\"obius inversion, this implies that $M_\alpha = \sum_{\beta \preceq \alpha} (-1)^{l(\beta) - l(\alpha)} L_\beta$.

For any labeled poset $\Pomega$, $\KPomega$ is a quasisymmetric function, and we can express it in terms of the fundamental basis $\{L_\alpha\}$ using the linear extensions of $\Pomega$. For any linear extension $\pi \in \mathcal L\Pomega$, define the \emph{descent set} of $\pi$ to be $\des(\pi) = \{i \mid \pi(i) > \pi(i+1)\}$. Abbreviating $\co(\des(\pi))$ by $\co(\pi)$, we then have the following result.

\begin{thm}[\cite{Gessel, Stanley3}]\label{L Expansion}
Let $(P, \omega)$ be a labeled poset.  Then
\[ \KPomega = \sum_{\pi \in \mathcal{L}\Pomega}L_{\co(\pi)}.\]
\end{thm}

\subsubsection{Quasisymmetric power sums}
%Formally, the \emph{type 1 quasisymmetric power sum basis} is the basis $\Psi_\alpha$ of $\QSym$ that satisfies $\langle \Psi_\alpha, \mathbf{\Psi}_\beta \rangle = z_\alpha\delta_{\alpha, \beta}$, where $\mathbf{\Psi}_\beta$ is the noncommutative power sum of the first kind, introduced in [].
Formally, the \emph{type 1 quasisymmetric power sum basis}, as defined in \cite{BallantineEtc}, is the basis $\{\Psi_\alpha\}$ of $\QSym$ that satisfies $\langle \Psi_\alpha, \mathbf{\Psi}_\beta \rangle =z_\alpha\delta_{\alpha, \beta}$, where $\mathbf{\Psi}_\beta$ is the noncommutative power sum of the first kind (introduced in \cite{NSym}).
  %The citation for this can be found on page 4 of the Ballantine, Daugherty, Hicks, Mason, and Niese paper.
The type 1 quasisymmetric power sum basis refines the power sum symmetric functions (see Section \ref{Murnaghan-Nakayama rule section}) as
\[p_{\lambda} = \sum_{\alpha \sim \lambda}\Psi_\alpha,\]
where the sum runs over all compositions $\alpha$ that rearrange to the partition $\lambda$.
We will consider the unnormalized version of the type 1 quasisymmetric power sum basis with basis elements $\{\psi_\alpha\}$ given by 
%\[\psi_\alpha = \frac{\Psi_\alpha}{z_\alpha}.\]
$\psi_\alpha = \frac{\Psi_\alpha}{z_\alpha}$.
From now on, when we refer to the type 1 quasisymmetric power sum basis, we are referring to the unnormalized version $\psi_\alpha$ unless stated otherwise.

We can express $\psi_\alpha$ in terms of the monomial basis:

\[ \psi_\alpha= \sum_{\beta \succeq \alpha} \frac{1}{\pi(\alpha, \beta)}M_\beta, \]
where the sum runs over all coarsenings $\beta$ of $\alpha$. This was proven in \cite{BallantineEtc}, but for the purposes of this paper, this can be taken as the definition of $\psi_\alpha$.

For more information on the type 1 quasisymmetric power sum basis, see \cite{AlexanderssonSulzgruber, BallantineEtc}.   %The citation for this can be found on page 4 of the Ballantine, Daugherty, Hicks, Mason, and Niese paper.

%In terms of the fundamental basis,
%\[ \Psi_\alpha= \frac{z_\alpha}{n!} \sum_{\gamma \succeq \alpha} |\{\sigma \in Cons_\alpha : \widehat{\alpha(\sigma)} = \gamma\}| \sum_{\eta \succeq \alpha^c} (-1)^{l(\eta)-1} L_{\gamma \vee \eta}, \]
%\color{red}{\textbf{ALL OF THIS NEEDS TO BE DEFINED AT SOME POINT}}\color{black}
%where $\gamma \vee \eta$ is the coarsest composition that refines both $\gamma$ and $\eta$.

\subsubsection{Length}
The type 1 quasisymmetric power sum basis has the following multiplicative property as shown in \cite{BallantineEtc}: %The citation is Theorem 5.3 of the Ballantine, Daugherty, Hicks, Mason, and Niese paper.
%\[ \psi_\alpha \psi_\beta = \sum_{\gamma \in \alpha \shuffle \beta} \psi_\gamma. \]
\[\psi_\alpha \psi_\beta = \sum\limits_{\gamma \in \alpha \shuffle \beta} \psi_\gamma.\]
Every composition $\gamma$ in the $\psi$-support of $\psi_\alpha \psi_\beta$ satisfies $l(\gamma) = l(\alpha) + l(\beta)$.  We can then use the $\psi$-basis to refine $\QSym_n$:
%\[\QSym_n = \bigoplus_{m \geq 0} \QSym_{n, m}.\]
\[\QSym_n = \bigoplus_{\lambda \vdash n} \QSym_{\lambda},\]
where $\QSym_{\lambda}$ is spanned by $\{\psi_\alpha \colon \alpha \sim \lambda \}$.  Observe that if $f \in \QSym_{\lambda}$ and $g \in \QSym_{\mu}$, then $f \cdot g \in \QSym_{\nu}$, where $\nu$ is the partition formed by combining and rearranging the parts of $\lambda$ and $\mu$.

We define $\QSym_{n, m}$ to be 
\[ \QSym_{n, m} := \bigoplus_{\substack{\lambda \vdash n \\ l(\lambda) = m}} \QSym_{\lambda}. \]
If $f \in \QSym_{n, m}$, then we say the \emph{length} of $f$ is $m$.  This gives a grading for $\QSym_n$,
\[\QSym_n = \bigoplus_{m \geq 0} \QSym_{n, m}.\]
If $f \in \QSym_{n_1, m_1}$ and $g \in \QSym_{n_2, m_2}$, then $f \cdot g \in \QSym_{n_1+n_2, m_1+m_2}$.

\subsubsection{Hopf Algebra}

%\color{red} Coproduct in L basis \color{black}
The ring of quasisymmetric functions is a Hopf algebra; in particular, it is equipped with a coproduct.  For more details on combinatorial Hopf algebras, see \cite{ABS, GrinbergReiner, Schmitt}.

The coproduct  is defined on the fundamental quasisymmetric function basis (see Proposition 5.2.15 in \cite{GrinbergReiner}) by
\[\Delta(L_{\alpha}) := \sum_{\substack{(\beta, \gamma) \\ \alpha = \beta \cdot \gamma \text{ or } \beta \odot \gamma}} L_{\beta} \otimes L_{\gamma}.\]
The coproduct on the type 1 quasisymmetric power sum basis satisfies
\[ \Delta (\psi_\alpha) = \sum_{\beta \cdot \gamma = \alpha} \psi_{\beta} \otimes \psi_\gamma.\]
This follows from the definition of the $\psi_\alpha$ in \cite{BallantineEtc}: the $\psi_\alpha$ are defined to be dual to the noncommutative power sums $\mathbf{\Psi}_\alpha$, whose multiplication is given by concatenation of compositions.

 %Other papers

We define the \textit{graded comultiplication} $\Delta_{\alpha}(\psi_\beta)$ to be
\[\Delta_{\alpha}(\psi_{\beta}) := \sum_{\substack{\gamma^{(1)} \cdots \gamma^{(l)} = \beta \\ \gamma^{(i)} \vDash \alpha_i}} \psi_{\gamma^{(1)}} \otimes \dots \otimes \psi_{\gamma^{(l)}},\]
where each $\gamma^{(i)}$ is a composition of $\alpha_i$.  (Therefore $\Delta_{\alpha}(\psi_{\beta}) = 0$ unless $\alpha \succeq \beta$.)  In other words, the graded coproduct gives one graded component of the iterated coproduct.  We define the graded comultiplication on the fundamental quasisymmetric function basis similarly.

From Theorem~\ref{L Expansion} and the coproduct on the fundamental basis described above, one can derive the following comultiplication formula for a $(P, \omega)$-partition generating function:
\[\Delta(\KPomega) = \sum_{I} K_{(I, \omega)}(\x)\otimes K_{(P \setminus I, \omega)}(\x),\]
where $I$ is an order ideal of $P$ (we abuse notation slightly by writing $\omega$ for the appropriate restriction of $\omega$ to either $I$ or $P \setminus I$).
Iterating and applying the graded comultiplication then gives
\begin{equation} \label{alpha coproduct} \tag{$*$}
\Delta_{\alpha}(\KPomega) = \sum K_{(P_1, \omega)}(\x) \otimes \cdots \otimes K_{(P_{l}, \omega)}(\x),
\end{equation}
%\[\Delta_{\alpha}(\KPomega) = \sum K_{(P_1, \omega)}(\x) \otimes \cdots \otimes K_{(P_{l}, \omega)}(\x),\]
where for each $i$, $|P_i| = \alpha_i$; $P_1, P_2, \dots, P_{l}$ partition $P$; and $P_1 \cup \dots \cup P_i$ is an order ideal of $P$.  %Note that we abuse notation slightly by writing $K_{(P_i, \omega)}(\x)$ since the $\omega$ that appears in $K_{(P_i, \omega)}(\x)$ is actually the restriction of $\omega$ to $P_i$.

%For more on the Hopf algebra structure of $\QSym$, see \cite{GrinbergReiner, Schmitt}.
%Let $\C[\mathcal{P}]$ be the free $\C$-module whose basis consists of isomorphism classes of finite posets.  It is well known (see \cite{Schmitt}) that $\C[\mathcal{P}]$ is a Hopf algebra with coproduct given by
%\[\Delta([P]) := \sum_{\text{ideal } I \subseteq P} [I] \otimes [P \setminus I],\]
%where the sum ranges over order ideals $I \subseteq P$.
%
%We define the  \textit{graded comultiplication} $\Delta_{\alpha}([P])$ to be
%\[ \Delta_{\alpha}([P]) := \sum [P_1] \otimes [P_2] \otimes \cdots \otimes [P_{l(\alpha})],\]
%where for each $i$, $|P_i| = \alpha_i$ and $P_1 \cup \dots \cup P_i$ is an order ideal of $P$.
%
%
%\[\Delta_{\alpha}(\KPomega) = \sum K_{(P_1, \omega)}(\x) \otimes K_{(P_2, \omega)}(\x) \otimes \cdots \otimes K_{(P_{l(\alpha)}, \omega)}(\x).\]
%
%These functions can likewise be defined on the free $\C$-module $\C[\mathcal{P}]$ whose basis consists of isomorphism classes of finite posets.  Explicitly:
%\begin{align*}
%\nabla([P_1] \otimes [P_2])&:= [P_1 \sqcup P_2],\\
%1_{\C[\mathcal{P}]} &:= \varnothing,\\
%\Delta([P]) &:= \sum_{\text{ideal } I \subseteq P} [I] \otimes [P \setminus I],\\
%\epsilon([P]) &:= \begin{cases}
%            1 & \text{if } |P|=0, \\
%            0 & \text{otherwise.}
%        \end{cases}
%\end{align*}

\section{Operations on $\QSym$}
In this section, we will express some useful linear functionals in terms of the type 1 quasisymmetric power sum basis.  We will then see how some well known automorphisms of $\QSym$ act on this basis.
\subsection{The $\MinOne$ and $\MaxOne$ functionals}
In \cite{LiuWeselcouch}, %first paper
the present authors showed the existence of linear functionals $\MinOne$ and $\MaxOne$ on $\QSym$ that satisfy 
\[
\MinOne(K_P(\mathbf{x})) = \begin{cases}
            1 & \text{if $P$ has exactly 1 minimal element,} \\
            0 & \text{otherwise,}
        \end{cases}
\]
and
\[
\MaxOne(K_P(\mathbf{x})) = \begin{cases}
            1 & \text{if $P$ has exactly 1 maximal element,} \\
            0 & \text{otherwise,}
        \end{cases}
\]
whenever $P$ is a naturally labeled poset.
(These functions were denoted $\min_1$ and $\max_1$ in \cite{LiuWeselcouch}.) 
On the fundamental quasisymmetric function basis, these functions act as follows:
\[
\MinOne(L_{\alpha}) = \begin{cases}
            (-1)^k & \text{if } \alpha = (1^{k}, n-k) \text{ for } 0 \leq k < n, \\
            0 & \text{otherwise,}
        \end{cases}
\]
and
\[
\MaxOne(L_{\alpha}) = 
	\begin{cases}
            (-1)^{k} &\text{if } \alpha = (n-k, 1^{k}) \text{ for } 0 \leq k < n, \\
            0 &\text{otherwise.}
        \end{cases}
\]

These functions can do more than just determine if a poset has exactly one minimal element or one maximal element: they can be used to test if a quasisymmetric function is irreducible.

\begin{lemma}\label{Max One product zero}
For all non-constant, homogeneous $f, g \in \QSym$, $\MinOne(f\cdot g) = \MaxOne(f \cdot g) = 0$.
\end{lemma}

\begin{proof}
The $P$-partition generating functions of naturally labeled posets span $\QSym$ \cite{KPSpanQSym}. %page 7
Therefore we can express $f$ and $g$ as a linear combination of these partition generating functions.  The product of the partition generating functions of any two naturally labeled posets gets sent to $0$ by $\MinOne$ since no disconnected poset has exactly $1$ minimal element.

A similar proof shows that $\MaxOne(f \cdot g) = 0$.

(One can also easily prove this result using the fundamental basis.)
\end{proof}

Note that if $f$ is a (non-constant, homogeneous) quasisymmetric function and $\MaxOne(f) \neq 0$ (or similarly if $\MinOne(f) \neq 0)$, then Lemma \ref{Max One product zero} tells us that $f$ is irreducible (and in fact does not lie in the span of homogeneous reducible elements of $\QSym$).
%What is significant about these functions is that they send the product of quasisymmetric functions to $0$.  Recall that if $P = Q \sqcup R$, then $\KP = \KQ \cdot K_R(\x)$, and note that the disjoint union of two posets cannot have exactly $1$ maximal element or exactly $1$ minimal element.  Therefore $\MinOne(\KQ \cdot K_R(\x)) = \MaxOne(\KQ \cdot K_R(\x)) = 0$ for all posets $Q$ and $R$.  The $P$-partition generating functions of naturally labeled posets span $\QSym$, so the product of any two quasisymmetric functions can be expressed as the product of the sum of $P$-partition generating functions of naturally labeled posets.  Therefore if $f, g \in \QSym$, then $\MinOne(f \cdot g) = \MaxOne(f \cdot g) = 0$.

It is straightforward to evaluate $\MinOne$ and $\MaxOne$ on the monomial basis.
\begin{lemma}\label{Mon Basis Expansion}
On the monomial basis $\{M_\alpha\}$,
\[\MaxOne(M_\alpha) = (-1)^{l(\alpha)-1}\alpha_1 \quad\text{ and }\quad
\MinOne(M_\alpha) = (-1)^{l(\alpha)-1}\alpha_{l(\alpha)}.\]
\end{lemma}

\begin{proof}
Expanding $M_{\alpha}$ in the $\{L_\alpha\}$ basis and applying $\MaxOne$ gives
\[\MaxOne(M_{\alpha}) = \sum_{\beta \preceq \alpha}(-1)^{l(\beta) -l(\alpha)}\MaxOne(L_\beta).\]
Since $\MaxOne(L_{\beta}) = 0$ unless $\beta = (n-k, 1^k)$ we have
\[\MaxOne(M_{\alpha}) = \sum_{i=0}^{\alpha_1 -1}(-1)^{n-\alpha_1+i+1-l(\alpha)}\MaxOne(L_{(\alpha_1-i, 1^{n-\alpha_1+i})}).\]
But $\MaxOne(L_{(\alpha_1-i, 1^{n-\alpha_1+i})}) = (-1)^{n-\alpha_1+i}$, therefore $\MaxOne(M_\alpha) = (-1)^{l(\alpha)-1}\alpha_1$ as desired.

A similar argument shows that $\MinOne(M_{\alpha}) = (-1)^{l(\alpha)-1}\alpha_{l(\alpha)}$.
\end{proof}

We can now evaluate the functions $\MaxOne$ and $\MinOne$ on the type 1 quasisymmetric power sum basis.

\begin{lemma} \label{MaxOne expansion}
On the type 1 quasisymmetric power sum basis $\{\psi_\alpha\}$,
\[\MaxOne(\psi_\alpha) = \sum_{i=1}^{l(\alpha)} \frac{(-1)^{l(\alpha)-i}}{\pi(\alpha_1\cdots\alpha_{i-1})\cdot \pi((\alpha_{i+1}\cdots\alpha_{l(\alpha)})^{\textrm{rev}})},\]
and
\[
\MinOne(\psi_{\alpha}) = \begin{cases}
            1 & \text{if } \alpha = (n), \\
            0 & \text{otherwise.}
        \end{cases}
\]
\end{lemma}

\begin{proof}
We first calculate $\MaxOne(\psi_\alpha)$. % = \sum_{i=1}^{l(\alpha)} \frac{(-1)^{l(\alpha)-i}}{\pi(\alpha_1\cdots\alpha_{i-1})\pi((\alpha_{i+1}\cdots\alpha_{l(\alpha)})^r)}$.
%Recall that we can express $\psi_\alpha$ in terms of the monomial basis and that $\MaxOne$ is linear.
Expanding  $\psi_\alpha$ in the monomial basis and applying $\MaxOne$ using Lemma~\ref{Mon Basis Expansion} gives
%\[ \psi_\alpha= \sum_{\beta \succeq \alpha} \frac{1}{\pi(\alpha, \beta)}M_\beta. \]
%From the linearity of $\MaxOne$, we can express $\MaxOne(\psi_{\alpha})$ as
\[ \MaxOne(\psi_{\alpha}) = \sum_{\beta \succeq \alpha} \frac{1}{\pi(\alpha, \beta)}\MaxOne(M_\beta) = \sum_{\beta \succeq \alpha} \frac{(-1)^{l(\beta)-1}\beta_1}{\pi(\alpha, \beta)}.\]
%Substituting the value of $\MaxOne(M_\beta)$ from Lemma \ref{Mon Basis Expansion}, we have 
%\[ \MaxOne(\psi_{\alpha}) = \sum_{\beta \succeq \alpha} \frac{(-1)^{l(\beta)-1}\beta_1}{\pi(\alpha, \beta)}.\]

Let $l(\alpha) = l$.  For each $\beta$ that coarsens $\alpha$, $\beta_1 = \alpha_1 + \cdots + \alpha_i$ for some value of $i$, and the composition $\gamma = (\beta_2, \dots, \beta_{l(\beta)})$ coarsens $\alpha^{i}:=(\alpha_{i+1}, \dots, \alpha_l)$.  We can then group together the compositions $\beta$ by their first component:
\begin{align*}
\MaxOne(\psi_\alpha) &= \sum_{i=1}^{l} \sum_{\gamma \succeq \alpha^{i}} \frac{(-1)^{l(\gamma)} (\alpha_1 + \cdots + \alpha_i)}{\pi(\alpha, (\alpha_1+ \cdots + \alpha_i) \cdot \gamma)}\\
&= \sum_{i=1}^{l} \sum_{\gamma \succeq \alpha^{i}} \frac{(-1)^{l(\gamma)}}{\pi(\alpha_1 \cdots \alpha_{i-1}) \cdot \pi(\alpha^{i}, \gamma)}.
\end{align*}
%\begin{align*}
%\MaxOne(\psi_{\alpha}) = &\sum_{\gamma \succeq \alpha^{(1)}} \frac{(-1)^{l(\gamma)}\alpha_1}{\pi(\alpha, (\alpha_1)\cdot \gamma)} +\sum_{\gamma \succeq \alpha^{(2)}} \frac{(-1)^{l(\gamma)}(\alpha_1+\alpha_2)}{\pi(\alpha, (\alpha_1+\alpha_2)\cdot \gamma)} + \dots + \sum_{\gamma \succeq \alpha^{(l)}} \frac{(-1)^{l(\gamma)}(\alpha_1+ \dots +\alpha_l)}{\pi(\alpha, (\alpha_1+ \dots +\alpha_l)\cdot \gamma)}\\
%= & \sum_{\gamma \succeq \alpha^{(1)}} \frac{(-1)^{l(\gamma)}}{\pi(\alpha^{(1)}, \gamma)} + \sum_{\gamma \succeq \alpha^{(2)}} \frac{(-1)^{l(\gamma)}}{\pi((\alpha_1))\pi(\alpha^{(2)}, \gamma)} + \dots + \sum_{\gamma \succeq \alpha^{(l)}} \frac{(-1)^{l(\gamma)}}{\pi((\alpha_1, \dots, \alpha_{l-1}))\pi(\alpha^{(l)} , \gamma)}.
%\end{align*}
By Lemma \ref{pi sum}, we have that, for all $i$,
\begin{align*}
 \sum_{\gamma \succeq \alpha^{i}} \frac{(-1)^{l(\gamma)}}{\pi(\alpha_1 \cdots \alpha_{i-1})\cdot \pi(\alpha^{i} , \gamma)} &= \frac{(-1)^{l(\alpha^{i})}}{\pi(\alpha_1\cdots \alpha_{i-1})}\sum_{\gamma \succeq \alpha^{i}} \frac{(-1)^{l(\alpha^{i})-l(\gamma)}}{\pi(\alpha^{i}, \gamma)}\\
&= \frac{(-1)^{l-i}}{\pi(\alpha_1\cdots\alpha_{i-1})\cdot \pi((\alpha_{i+1}\cdots\alpha_{l})^{\textrm{rev}})},
 \end{align*}
as desired.

We will now show that $\MinOne(\psi_{\alpha})= 0$ unless $\alpha = (n)$, in which case $\MinOne(\psi_{(n)})= 1$.
As before, we will express $\psi_\alpha$ in terms of the monomial basis, and evaluate $\MinOne$ on both sides.  This gives 
\[ \MinOne(\psi_{\alpha}) = \sum_{\beta \succeq \alpha} \frac{1}{\pi(\alpha, \beta)}\MinOne(M_\beta) = \sum_{\beta \succeq \alpha} \frac{(-1)^{l(\beta)-1}\beta_{l(\beta)}}{\pi(\alpha, \beta)}.\]
When $l(\alpha) > 1$, for each $\gamma$ that coarsens $(\alpha_1, \dots, \alpha_{l-1})$, both $\beta = \gamma \cdot (\alpha_l)$ and $\beta' = \gamma \odot (\alpha_l)$ are coarsenings of $\alpha$.
Therefore 
\[ \sum_{\beta \succeq \alpha} \frac{(-1)^{l(\beta)-1}\beta_{l(\beta)}}{\pi(\alpha, \beta)} = \sum_{\gamma \succeq (\alpha_1, \dots, \alpha_{l-1})} \left(\frac{(-1)^{l(\gamma)}\alpha_l}{\pi(\alpha, \gamma \cdot (\alpha_l))} + \frac{(-1)^{l(\gamma)-1}(\gamma_{l(\gamma)}+\alpha_l)}{\pi(\alpha, \gamma \odot (\alpha_l))}\right).\]
But for all $\gamma$, 
%\[\frac{(-1)^{l(\gamma)}\alpha_l}{\pi(\alpha, (\gamma \cdot \alpha_l))} + \frac{(-1)^{l(\gamma)-1}(\gamma_{l(\gamma)}+\alpha_l)}{\pi(\alpha, (\gamma \odot \alpha_l))} = 0.\]
\[\pi(\alpha, \gamma \cdot (\alpha_l)) = \pi((\alpha_1, \dots, \alpha_{l-1}), \gamma)\cdot \alpha_l,\] 
and 
\[\pi(\alpha, \gamma \odot (\alpha_l)) = \pi((\alpha_1, \dots \alpha_{l-1}), \gamma)\cdot (\gamma_{l(\gamma)} + \alpha_l).\]
Therefore each term in the sum vanishes, so $\MinOne(\psi_\alpha) = 0$ when $l(\alpha) > 1$.  When $\alpha = (n)$, we have
\[\MinOne(\psi_{(n)})= \frac{n}{\pi((n), (n))} = 1.\qedhere\]
\end{proof}

\begin{example}
Let $\alpha = 3421$.  We have 
\[\MaxOne(\psi_{3421}) = \frac{(-1)^3}{(1 \cdot 3 \cdot 7)} + \frac{(-1)^2}{(3)(1 \cdot 3)} + \frac{(-1)^1}{(3\cdot 7)(1)} + \frac{(-1)^0}{(3 \cdot 7 \cdot 9)} = \frac{4}{189}.\]
\end{example}

Due to the simplicity of the behavior of $\MinOne$ on $\psi_\alpha$, we can now compute the coefficients in the $\psi_\alpha$-expansion of any quasisymmetric function.  
\begin{thm}\label{coeff of psi alpha}
Suppose $f \in \QSym$ and $f = \sum_\alpha c_\alpha \psi_\alpha$.  Then
\[c_\alpha = \MinOne^{\otimes l(\alpha)}(\Delta_{\alpha}f).\]
%where $\MinOne^{l(\alpha)} = \MinOne \otimes \dots \otimes \MinOne$.
\end{thm}

\begin{proof}
Since $\MinOne$ is linear, we can prove this by showing that  $\MinOne^{\otimes l(\alpha)}(\Delta_{\alpha}\psi_\beta) = \delta_{\alpha, \beta}$.
Recall that 
\[\Delta_{\alpha}(\psi_{\beta}) = \sum_{\substack{\gamma^{(1)} \cdots \gamma^{(l)} = \beta \\ \gamma^{(i)} \vDash \alpha_i}} \psi_{\gamma^{(1)}} \otimes \dots \otimes \psi_{\gamma^{(l)}}\]
and $\MinOne(\psi_\alpha) = \delta_{l(\alpha), 1}$.  The only way compositions with length $1$ can concatenate to $\beta$ is if the compositions are $(\beta_1)$, $(\beta_2), \dots$.  Since then $\beta_i = \alpha_i$ for all $i$, it follows that $\MinOne^{\otimes l(\alpha)}(\Delta_{\alpha}\psi_\beta) = \delta_{\alpha, \beta}$.
\end{proof}

It should be noted that Theorem \ref{coeff of psi alpha} follows immediately from the combinatorial description of the coefficient of $\psi_\alpha$ in $\KP$ given by Alexandersson and Sulzgruber \cite{AlexanderssonSulzgruber}. Indeed, consider the following definition of a pointed $P$-partition.

\begin{defn}
Let $P$ be a naturally labeled poset. A $P$-partition $\theta$ is \emph{pointed} if $\theta$ is surjective onto $[k]$ for some $k$, and $\theta^{-1}(i)$ has a unique minimal element for all $i \in [k]$.
\end{defn}

We say that the \emph{weight} of a pointed $\Pomega$-partition $\theta$ is the composition \[\wt(\theta) = (|\theta^{-1}(1)|, |\theta^{-1}(2)|, \dots ).\] It is shown in \cite{AlexanderssonSulzgruber} that the coefficient of $\psi_\alpha$ in $\KP$ is the number of pointed $P$-partitions with weight $\alpha$.  This is the same as evaluating $\MinOne$ on each factor in the graded coproduct of $\KP$.  We will state this as a corollary to Theorem \ref{coeff of psi alpha}.

%They showed that the coefficient of $\psi_\alpha$ in $\KP$ is the number of surjective $P$-partitions $\phi \colon P \rightarrow [k]$, such that $\phi^{-1}(i)$ has a unique minimal element for all $i \in [k]$.  This is equivalent to evaluating $\MinOne$ to the graded coproduct of $P$.
%\color{red}\textbf{This last sentence needs to be more accurate}\color{black}

\begin{cor}[\cite{AlexanderssonSulzgruber}, Theorem 5.4]\label{Naturally Labeled Expansion}
Let $P$ be a naturally labeled poset.  Then
\[\KP = \sum_\theta \psi_{\wt(\theta)},\]
where the sum runs over all pointed $P$-partitions $\theta$.
\end{cor}
\begin{proof}
%This is Theorem 5.4 in \cite{AlexanderssonSulzgruber}.
Suppose $\KP = \sum_\alpha c_\alpha \psi_\alpha$.  By Theorem \ref{coeff of psi alpha}, $c_\alpha = \MinOne^{\otimes l(\alpha)}(\Delta_{\alpha}\KP)$.  Recall that $\MinOne(\KP) = 1$ if $P$ has exactly one minimal element and $\MinOne(\KP) = 0$ otherwise.  From equation \eqref{alpha coproduct} above, this means that $\MinOne^{\otimes l(\alpha)}(\Delta_{\alpha}\KP)$ is the number of ways to partition $P$ into $P_1, P_2, \dots, P_{l(\alpha)}$ where for all $i$, $|P_i| = \alpha_i$, $P_1 \cup \dots \cup P_i$ is an order ideal of $P$, and $P_i$ has exactly one minimal element.  This is exactly the number of pointed $P$-partitions with weight $\alpha$.
\end{proof}

In Section $6$ we will extend this result to all labeled posets.

\subsection{Automorphisms}
%Suppose that $\Pomega$ is a labeled poset.  Let $P^*$ be the poset formed by reversing the relations of $P$, that is $x \preceq_P y \iff y \preceq_{P^*} x$.  Let $\omega^*$ be the labeling of $P^*$ defined by $\omega^*(x) = n - \omega(x) + 1$.

For this section, suppose that $(P, w)$ is a labeled poset.  Let $P^*$ be the poset formed by reversing the relations of $P$, that is $x \preceq_P y \iff y \preceq_{P^*} x$.  Let $w^*$ be the labeling of $P$ or $P^*$ defined by $w^*(x) = n - w(x) + 1$.

%\color{red}\textbf{Should I not use $\omega$ for this automorphism? }\color{black}

We will consider three well-known automorphisms of $\QSym$: $\omega$, $\rho$, and their composition $\omega\rho$. (In \cite{LuotoMykytiukVanWIlligenburg}, $\omega\rho$ is denoted by $\psi$.)  These automorphisms act as follows on the fundamental basis $\{L_\alpha\}$:
\begin{align*}
\omega(L_\alpha) &= L_{(\alpha^c)^{\textrm{rev}}},\\
\rho(L_\alpha) &= L_{\alpha^{\textrm{rev}}},\\
\omega\rho(L_\alpha) &= L_{\alpha^c}.
\end{align*}
Here $\alpha^c$ is the composition such that $D(\alpha)$ and $D(\alpha^c)$ are complementary subsets of $[n-1]$.

These automorphisms perform the following actions on $\KPomega$:
\begin{align*}
\omega(K_{(P,w)}(\x)) &= K_{(P^*,w)}(\x),\\
\rho(K_{(P,w)}(\x)) &= K_{(P^*,w^*)}(\x),\\
\omega\rho(K_{(P,w)}(\x)) &= K_{(P,w^*)}(\x).
\end{align*}
Informally, in terms of the Hasse diagram of $P$, $\rho$ flips $P$ upside down, $\omega\rho$ switches natural edges with strict edges, and $\omega$ does both.
%Since all three of these automorphisms commute and are self inverses, we can express one as the composition of the other two.
Each of these automorphism can be expressed as the composition of the other two.
%Since we are using the letter $\psi$ for basis elements, to avoid confusion, we will use the notation $\omega\rho$ for $\psi$.  
For more on these automorphisms, see Section 3.6 in \cite{LuotoMykytiukVanWIlligenburg}. %http://www.math.ubc.ca/%7Esteph/papers/QuasiSchurBook.pdf

%On the fundamental quasisymmetric function basis, $\omega(L_\alpha) = L_{\alpha^{t}}$.    %Need to define alpha^t, the composition that is the transposition of alpha
The authors of \cite{BallantineEtc} show that $\omega(\psi_\alpha) = (-1)^{|\alpha| - l(\alpha)}\psi_{\alpha^{\textrm{rev}}}$ (though they do not give an expansion for the other two automorphisms). This result is easy to deduce given the earlier results in this section, as we now show.
\begin{thm} [\cite{BallantineEtc}] \label{prop:omega}
For any composition $\alpha$, $\omega(\psi_\alpha) = (-1)^{|\alpha| - l(\alpha)} \psi_{\alpha^{\textrm{rev}}}$.
\end{thm}
\begin{proof}
If $\alpha = (1^k,n-k)$, then $(\alpha^{c})^{\textrm{rev}} = (1^{n-k-1}, k+1)$. Hence, for all compositions $\alpha$, $\MinOne \circ \omega(L_\alpha) = (-1)^{|\alpha|-1} \MinOne(L_\alpha)$. In other words, $\MinOne \circ \omega$ acts like $(-1)^{n-1} \MinOne$ on $\QSym_n$.

By Theorem~\ref{coeff of psi alpha}, the coefficient of $\psi_\beta$ in $\omega(\psi_\alpha)$ is $\MinOne^{\otimes l(\beta)} \Delta_{\beta} (\omega(\psi_\alpha))$. From the action of $\omega$ on the fundamental basis $L_\alpha$ as given above, we see that 
%$\Delta_\beta \circ \omega$ acts by complementing and reversing $\alpha$ and then breaking it into compositions whose sizes are given by the parts of $\beta$. This is equivalent to first breaking $\alpha$ into compositions whose sizes are given by the parts of $\beta^{\mathrm{rev}}$, then complementing and reversing each piece and reversing all the pieces. In other words,
\[\Delta_\beta \circ \omega = \rev(\omega^{\otimes l(\beta)} \circ \Delta_{\beta^{\textrm{rev}}}),\]
where $\rev$ is the automorphism of $\QSym^{\otimes l(\beta)}$ that reverses tensor factors. 
Thus since $\MinOne^{\otimes l(\beta)} \circ \rev = \MinOne^{\otimes l(\beta)}$,
\[\MinOne^{\otimes l(\beta)} \Delta_{\beta} (\omega(\psi_\alpha)) = (\MinOne \circ \omega)^{\otimes l(\beta)} \Delta_{\beta^{\textrm{rev}}}(\psi_\alpha).\]
This is only nonzero if $\beta = \alpha^{\textrm{rev}}$, in which case it equals $\prod_i (-1)^{|\alpha_i|-1} = (-1)^{|\alpha| - l(\alpha)}$.
\end{proof}

%It is well known that on the fundamental quasisymmetric function basis, $\rho(L_\alpha) = L_{\alpha^r}$.
Since $\MinOne(L_\alpha) = \MaxOne(L_{\alpha^{\textrm{rev}}})$, it follows that $\MinOne \circ \rho = \MaxOne$.  This allows us to give an expansion for $\omega\rho(\psi_\alpha)$ in the type 1 quasisymmetric power sum basis.
\begin{thm}
Let $\rho(\psi_\alpha) = \sum_{\beta} c_\beta \psi_\beta$. Then
\[c_\beta = \MaxOne^{\otimes l(\beta)} \Delta_{\beta^{\mathrm{rev}}}(\psi_\alpha).\]
\end{thm}
\begin{proof}
Using the fact that $\MinOne \circ \rho = \MaxOne$, we have from Theorem~\ref{coeff of psi alpha} that
\[c_\beta = \MinOne^{\otimes l(\beta)}\Delta_{\beta}(\rho(\psi_\alpha)) = (\MinOne \circ \rho)^{\otimes l(\beta)} \Delta_{\beta^{\textrm{rev}}} (\psi_\alpha) = \MaxOne ^{\otimes l(\beta)} \Delta_{\beta^{\textrm{rev}}} (\psi_\alpha).\]
Here we use the fact that, as with $\omega$ in Theorem~\ref{prop:omega}, $\Delta_\beta \circ \rho = \rev(\rho^{\otimes l(\beta)} \circ \Delta_{\beta^{\textrm{rev}}})$.
\end{proof}

Combining these results immediately gives the following theorem.
\begin{thm}
Let $\omega\rho(\psi_\alpha) = \sum_\beta c_\beta \psi_\beta$. Then 
\[ c_\beta = (-1)^{|\beta| - l(\beta)}\MaxOne^{\otimes l(\beta)}\Delta_{\beta}(\psi_\alpha).\]
\end{thm}

%\begin{proof}
%From Theorem \ref{coeff of psi alpha}, we know that $c_\beta = \MinOne^{\otimes l(\beta)}\Delta_{\beta}(\omega\rho(\psi_\alpha)).$  
%Using the fact that $\MinOne \circ \rho = \MaxOne$ and $\omega(\psi_\alpha) = (-1)^{|\alpha| - l(\alpha)}\psi_{\alpha^r}$, we have from Theorem \ref{coeff of psi alpha} that
%\[c_\beta = \MinOne^{\otimes l(\beta)}\Delta_{\beta}(\omega\rho(\psi_\alpha)) = (-1)^{|\alpha| - l(\alpha)}\MaxOne^{\otimes l(\beta)}\Delta_{\beta^r}(\psi_{\alpha^r}). \]  By Lemma \ref{MaxOne expansion}, we have $\MaxOne(\psi_{\alpha^r}) = (-1)^{l(\alpha)-1}\MaxOne(\psi_{\alpha})$.  Therefore
%\[ (-1)^{|\alpha| - l(\alpha)}\MaxOne^{\otimes l(\beta)}\Delta_{\beta^r}(\psi_{\alpha^r}) = (-1)^{|\alpha| - l(\alpha)}(-1)^{l(\alpha) - l(\beta)}\MaxOne^{l(\beta)}\Delta_{\beta}(\psi_\alpha), \]
%as desired.
%Using the fact that $\MinOne \circ \rho = \MaxOne$, we have from Theorem~\ref{coeff of psi alpha} that
%\[c_\beta = \MinOne^{\otimes l(\beta)}\Delta_{\beta}(\omega\rho(\psi_\alpha)) = (\MinOne \circ \omega \rho)^{\otimes l(\beta)} \Delta_{\beta} (\psi_\alpha) = (\MaxOne \circ \omega)^{\otimes l(\beta)} \Delta_{\beta} (\psi_\alpha).\]
%As in the proof of Proposition~\ref{prop:omega}, $\MaxOne \circ \omega$ only differs from $\MaxOne$ by a sign, so
%\[c_\beta = (-1)^{|\beta| - l(\beta)} \MaxOne^{\otimes l(\beta)} \Delta_\beta (\psi_\alpha),\] as desired.
%\end{proof}

In particular, $\psi_{\beta}$ can only appear in the expansion of $\rho(\psi_\alpha)$ if $\beta^{\textrm{rev}}$ coarsens $\alpha$, and likewise $\psi_{\beta}$ can only appear in the expansion of $\omega\rho(\psi_{\alpha})$ if $\beta$ coarsens $\alpha$. Of course, these coefficients can be computed explicitly using Lemma~\ref{MaxOne expansion}.

\begin{example}
By the previous theorem, the expansion of $\omega\rho(\psi_{3421})$ in the $\psi_\alpha$ basis is
\[\omega\rho(\psi_{3421}) = \psi_{3421} + \frac{1}{2}\psi_{343} + \frac{1}{4}\psi_{361} + \frac{1}{8}\psi_{37} - \frac{1}{12}\psi_{721} - \frac{1}{24}\psi_{73} - \frac{1}{28}\psi_{91} - \frac{4}{189}\psi_{10}.\]
\end{example}

\section{Irreducibility of $\KP$}
In this section, we will restrict our attention to the case when $P$ is naturally labeled. We will show that in this case, $K_P(\x)$ is irreducible whenever $P$ is connected, which partially answers a question from \cite{McNamaraWard}.  %The proof is rather long so for clarity we will now present an outline of the steps.

We first give a brief outline of the proof. We begin by showing that $\KP$ is irreducible whenever the truncation $\KPmin$ of $\KP$ to its terms of shortest length is irreducible, and we give a combinatorial description of the $\psi_\alpha$-expansion of $\KPmin$.  We then use Lemma~\ref{MaxOne expansion} to evaluate $\MaxOne(\KPmin)$ as an alternating sum. By interpreting the terms in this sum combinatorially using inclusion-exclusion, we show that $\MaxOne(\KPmin)$ counts certain zigzag labelings of $P$ that exist if and only if $P$ is connected. In particular, $\MaxOne(\KPmin) > 0$ when $P$ is connected, so by Lemma~\ref{Max One product zero}, we can then conclude our result.

As a remark, it was shown in \cite{LamPylyavskyy} that a homogeneous quasisymmetric function is reducible in $\QSym$ if and only if it is reducible in the ring $\mathbb C[[x_1, x_2, \dots]]$ of formal power series of bounded degree. (Though the authors of \cite{LamPylyavskyy} work over $\mathbb Z$ instead of $\mathbb C$, the results of that paper as well as this one hold in either context.) Note also that by standard degree considerations, if a homogeneous quasisymmetric function is reducible, then its irreducible factors must be homogeneous as well.

\subsection{Minimal length}
%Let $\tau_{m} \colon \QSym_n \rightarrow \QSym_{n, m}$ be the projection map onto the length $m$ part; in other words,
%\[
%\tau_{m}(\psi_{\alpha}) = 
%	\begin{cases}
%            \psi_\alpha &\text{if } l(\alpha) = m, \\
%            0 &\text{otherwise}.
%        \end{cases}
%\]
%Given a poset $P$, let $\KPmin := \tau_m(\KP)$, where $m$ is the minimum length composition in the $\psi_\alpha$-expansion of $\KP$. In other words, $\KPmin$ consists of the terms in the expansion of $\KP$ in the $\psi_\alpha$-basis of minimal length. As we will see shortly, $m$ will equal the number of minimal elements of $P$.
Given a poset $P$ of size $n$, let $m$ be the minimum length of a composition in the $\psi_\alpha$-expansion of $\KP$. (As we will see shortly, $m$ will equal the number of minimal elements of $P$.) Then define $\KPmin$ to be the sum of the terms in the $\psi_\alpha$-expansion of $\KP$ of this minimum length $m$. In other words, $\KPmin$ is the projection of $\KP$ to $\QSym_{n,m}$.

%Suppose that $P$ is an $n$ element poset with exactly $m$ minimal elements.  Denote by $\KPmin$ the image of $\KP$ under $\tau_{m}$, that is $\KPmin := \tau_{m}(\KP)$.  
%Informally, $\KPmin$ is the quasisymmetric function formed by taking only the terms in the expansion of $\KP$ in the $\psi_\alpha$-basis whose composition has minimal length in $\KP$.%That is to say that if $\KP = \sum_{\alpha}c_{\alpha}\Psi_\alpha$ and $m$ is the length of the shortest composition in the support of $\KP$, then $\KPmin  = \sum_{\alpha, |\alpha|=m}c_{\alpha}\Psi_\alpha$.  From the combinatorial definition of the values of $c_\alpha$, it is clear that $m$ is the number of minimal elements of $P$. 

\begin{example}  Let $P$ be the following $5$-element poset.
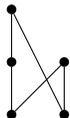
\[
\begin{tikzpicture}
 [auto,
 vertex/.style={circle,draw=black!100,fill=black!100,thick,inner sep=0pt,minimum size=1mm}, scale = .7]
\node (v1) at ( 0,0) [vertex] {};
\node (v2) at ( 0,1) [vertex] {};
\node (v3) at (0,2) [vertex] {};
\node (v4) at ( 1,0) [vertex] {};
\node (v5) at ( 1,1) [vertex] {};
\draw [-] (v1) to (v2);
\draw [-] (v2) to (v3);
\draw [-] (v4) to (v5);
\draw [-] (v1) to (v5);
\draw [-] (v4) to (v3);
\end{tikzpicture}
\]
Then
\begin{align*}
\KP &= \psi_{23} + \psi_{14}  + 2\psi_{221}  + 2\psi_{131} + 2\psi_{122} + \psi_{113} \\
& \quad + 2\psi_{2111} + 4\psi_{1211} + 4\psi_{1121} + 2\psi_{1112} + 8\psi_{11111}, \\
\KPmin &= \psi_{23} + \psi_{14}.
\end{align*}
\end{example}

Just as with the usual notion of degree, we can check for irreducibility of a quasisymmetric function by checking irreducibility on the part of minimal length.

\begin{lemma}\label{Reducible KP}
If $\KP$ is reducible, then $\KPmin$ is also reducible.
\end{lemma}
\begin{proof}
Recall that multiplication in terms of the type 1 quasisymmetric power sum basis is the shuffle product, so length gives a grading on $\QSym$. If $\KP$ is reducible, then $\KP$ can be expressed as $\KP = f\cdot g$ for some nonconstant homogeneous $f,g \in \QSym$. If $\tilde{f}$ and $\tilde{g}$ are the terms of shortest length in $f$ and $g$, respectively, then $\KPmin = \tilde{f} \cdot \tilde{g}$.
\end{proof}
%\cite{Pylyavskyy}
%We will use the contrapositive of Lemma \ref{Reducible KP} that says 
It follows that if $\KPmin$ is irreducible, then $\KP$ must also be irreducible.
%
%
%For $n \geq 0$, consider the function $\MaxOne\colon \QSym_n \rightarrow \Z^+$ defined on the monomial quasisymmetric function basis by
%\[
%\MaxOne(M_{\alpha}) = (-1)^{l(\alpha)-1}\alpha_1. \]
%	\begin{cases}
%            (-1)^{k} &\text{if } \alpha = (n-k-1)\odot 1^{k+1} \text{ for } 0 \leq k < n, \\
%            0 &\text{otherwise.}
%        \end{cases}
%\]
%This is the same function as $\max_1$ from [].  When applied to the $P$-partition generating function of a naturally labeled poset, $\MaxOne(\KP) = 1$ whenever $P$ has exactly $1$ maximal element, and $\MaxOne(\KP) = 0$ otherwise.  In particular, $\MaxOne(\KP) = 0$ whenever $P$ is disconnected since disconnected posets have at least $2$ maximal elements.
%

We now give a combinatorial interpretation for $\KPmin$ in terms of certain special pointed $P$-partitions. 
Let $\{z_1, \dots, z_m\}$ be the set of minimal elements of $P$.  For any subset $\{i_1, \dots, i_l\} \subseteq [m]$, we will denote by $\V(z_{i_1}, \dots, z_{i_l})$ the filter of $P$ whose minimal elements are $\{z_{i_1}, \dots, z_{i_l}\}$.

%\begin{defn}
%Suppose $P$ has $m$ minimal elements and $\sigma \in S_m$.  We say that a pointed $P$-partition $f$ is a \emph{$\sigma$-partition} if $f$ is the unique $P$-partition that sends the principal filter of $z_{\sigma_m}$ in $P$ to $m$, and for all $i < m$, $f$ sends the principal filter of $z_{\sigma_i}$ in $P \setminus V_{\{z_{\sigma_{i+1}}, \dots, z_{\sigma_{m}}\}}$ to $i$.  Denote by $\alpha(\sigma)$ the weight of the $\sigma$-partition.
%\end{defn}

\begin{defn}
Suppose $P$ has $m$ minimal elements and $\sigma \in S_m$. The \emph{$\sigma$-partition} $\theta_\sigma$ is the pointed $P$-partition that sends any element $x \in P$ to the largest number $j$ such that $z_{\sigma_j} \preceq x$. In other words, $\theta_\sigma^{-1}(i) = \V(z_{\sigma_i}) \setminus \V(z_{\sigma_{i+1}}, \dots, z_{\sigma_{m}})$ for $i = 1, \dots, m$.
\end{defn}

Denote by $\alpha(\sigma)$ the weight of $\theta_\sigma$. Explicitly,
\begin{align*}
\alpha(\sigma)_i &= |\V(z_{\sigma_i}) \setminus \V(z_{\sigma_{i+1}}, \dots, z_{\sigma_m})|\\
&= |\V(z_{\sigma_i}, \dots, z_{\sigma_m})| - |\V( z_{\sigma_{i+1}}, \dots, z_{\sigma_m})|.
\end{align*}

\begin{example}\label{Ex sigma partition}
Let $P$ be the poset shown below (on the left). % whose minimal elements $z_1, z_2, z_3$ are drawn from left to right.
On the right is the $\sigma$-partition of $P$ when $\sigma = 312$ with weight $\alpha(\sigma) = (1,2,5)$.
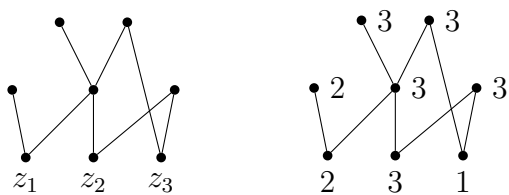
\[
\begin{tikzpicture}
 [auto,
 vertex/.style={circle,draw=black!100,fill=black!100,thick,inner sep=0pt,minimum size=1mm}, scale = .9]
\node (v1) at ( 0,0) [vertex, label=below:$z_2$] {};
\node (v2) at ( 0,1) [vertex] {};
\node (v3) at (.5,2) [vertex] {};
\node (v4) at ( 1,0) [vertex, label=below:$z_3$] {};
\node (v5) at ( 1.2,1) [vertex] {};
\node (v6) at ( -1,0) [vertex, label=below:$z_1$] {};
\node (v7) at (-1.2, 1) [vertex] {};
\node (v8) at (-.5, 2) [vertex] {};
\draw [-] (v1) to (v2);
\draw [-] (v2) to (v3);
\draw [-] (v4) to (v5);
\draw [-] (v1) to (v5);
\draw [-] (v4) to (v3);
\draw [-] (v6) to (v2);
\draw [-] (v6) to (v7);
\draw [-] (v2) to (v8);
\end{tikzpicture}
\qquad\qquad
\begin{tikzpicture}
[auto,
vertex/.style={circle,draw=black!100,fill=black!100,thick,inner sep=0pt,minimum size=1mm}, scale = .9]
\node (v1) at ( 0,0) [vertex, label=below:$3$] {};
\node (v2) at ( 0,1) [vertex, label=right:$3$] {};
\node (v3) at (.5,2) [vertex, label=right:$3$] {};
\node (v4) at ( 1,0) [vertex, label=below:$1$] {};
\node (v5) at ( 1.2,1) [vertex, label=right:$3$] {};
\node (v6) at ( -1,0) [vertex, label=below:$2$] {};
\node (v7) at (-1.2, 1) [vertex, label=right:$2$] {};
\node (v8) at (-.5, 2) [vertex, label=right:$3$] {};
\draw [-] (v1) to (v2);
\draw [-] (v2) to (v3);
\draw [-] (v4) to (v5);
\draw [-] (v1) to (v5);
\draw [-] (v4) to (v3);
\draw [-] (v6) to (v2);
\draw [-] (v6) to (v7);
\draw [-] (v2) to (v8);
\end{tikzpicture}
\]
%The following is the $\sigma$-partition of $P$ when $\sigma = 312$.  Its weight is $\alpha(\sigma) = (1,2,5)$.

%It follows from Lemma \ref{KPmin Expansion} that
%\[\KPmin = \psi_{233} + \psi_{215} + \psi_{143} + 3\psi_{125}.\]
\end{example}

%\[D(\alpha(\sigma)^r) = \{|\V(z_{\sigma_m})|, \quad |\V(z_{\sigma_{m-1}}, z_{\sigma_m})|, \quad \dots, \quad |\V(z_{\sigma_{2}}, \dots, z_{\sigma_m})|  \}.\]

%For $\sigma \in S_m$, let $\alpha(\sigma)$ be the following composition of $n$,
%\[\alpha(\sigma) = (\co(\{|\V_{ \{ z_{\sigma_m} \} }|, |\V_{ \{ z_{\sigma_{m-1}}, z_{\sigma_m} \} }|, \dots, |\V_{\{z_{\sigma_{2}}, \dots, z_{\sigma_m} \}}|  \}))^r.\]
%This is the weight of the pointed $P$-partition $f$ that sends that sends the principal filter of $z_{\sigma_m}$ in $P$ to $m$, and for all $i < m$, $f$ sends the principal filter of $z_{\sigma_i}$ in $P \setminus V_{\{z_{\sigma_{i+1}}, \dots, z_{\sigma_{m}}\}}$ to $i$.

\begin{lemma}\label{KPmin Expansion}
Let $P$ be a poset with $m$ minimal elements.  Then
\[ \KPmin = \sum_{\sigma \in S_m} \psi_{\alpha(\sigma)} .\]%= \sum_{\sigma \in S_m} P_{\alpha(\sigma)} .\]
\end{lemma}

\begin{proof}
This follows from Corollary \ref{Naturally Labeled Expansion}. Note that in any pointed $P$-partition, the minimal elements of $P$ must be sent to different values. Hence the minimum length part of $\KPmin$ has length (at least) $m$.

If $\theta \colon P \to [m]$ is a pointed $P$-partition, then there exists some permutation $\sigma$ such that $\theta(z_{\sigma_i}) = i$ for all $i$. Since $\theta^{-1}(m)$ must be a principal filter of $P$, it must be $V(z_{\sigma_m})$. Similarly, $\theta^{-1}(m-1)$ must be a principal filter of $P \setminus V(z_{\sigma_m})$, so it must be $V(z_{\sigma_{m-1}}) \setminus V(z_{\sigma_m})$. Continuing in this manner, we see that we must have $\theta = \theta_\sigma$, and the result follows.
%The combinatorial description of the coefficients of $\KPmin$ can be expressed as a sum over all pointed $P$-partitions that surject onto $[m]$.  Let $f$ be the $P$-partition that sends the principal filter of $z_{\sigma_m}$ in $P$ to $m$, and for all $i < m$, $f$ sends the principal filter of $z_{\sigma_i}$ in $P \setminus \V({\{z_{\sigma_{i+1}}, \dots, z_{\sigma_{m}}\}})$ to $i$.  This map is a pointed $P$-partition.
\end{proof}

%\begin{defn}
%Suppose $P$ has $m$ minimal elements and $\sigma \in S_m$.  We say that a pointed $P$-partition $f$ is a \emph{$\sigma$-partition} if $f$ is the unique $P$-partition that sends the principal filter of $z_{\sigma_m}$ in $P$ to $m$, and for all $i < m$, $f$ sends the principal filter of $z_{\sigma_i}$ in $P \setminus V_{\{z_{\sigma_{i+1}}, \dots, z_{\sigma_{m}}\}}$ to $i$.  
%\end{defn}
%
%Observe that the weight of a $\sigma$-partition is $\alpha(\sigma)$.

\begin{example}
Let $P$ be the poset shown in Example \ref{Ex sigma partition}. One can compute for all $\sigma \in S_3$:
\begin{align*}
\alpha(123) &= (2,3,3),& \alpha(132) &=  (2,1,5),&\alpha(213) &= (1,4,3),\\
 \alpha(231) &= (1,2,5),&\alpha(312) &= (1,2,5),&\alpha(321) &= (1,2,5).
 \end{align*}
It follows from Lemma \ref{KPmin Expansion} that
\[\KPmin = \psi_{233} + \psi_{215} + \psi_{143} + 3\psi_{125}.\]
\end{example}

\subsection{Evaluation of $\MaxOne$}
We will now give a combinatorial interpretation for the value of $\MaxOne$ when applied to some term $\psi_{\alpha(\sigma)}$ appearing in $\KPmin$. To do this, we will need the following technical notion.

\begin{defn}
Let $\pi, \sigma \in S_m$ and let $\{z_1, \dots, z_m\}$ be the set of minimal elements of a poset $P$.  We say that a bijection $\phi \colon P \to [n]$ is a \emph{$(\pi, \sigma)$-labeling} if the following are true:
\begin{enumerate}
\item $\phi(z_{\pi_1}) > \phi(z_{\pi_2}) > \dots > \phi(z_{\pi_m})$; and
%\item For all $y \in P$, $\phi(y) \geq \phi(z_{\sigma_{j(y)}})$ where $j(y)$ is the largest number such that $y \succeq z_{\sigma_{j(y)}}$.
\item for all $y \in P$, $\phi(y) \geq \phi(z)$, where $z$ is the minimal element such that $\theta_\sigma(z) = \theta_\sigma(y)$.
%\phi(z_{\sigma_{\theta_\sigma(y)}})$ where $\theta_\sigma$ is the $\sigma$-partition of $P$.
\end{enumerate}
Denote by $T_P(\pi, \sigma)$ the set of all $(\pi,\sigma)$-labelings of $P$.% and let $T_P(\pi, \sigma) = |t_P(\pi, \sigma)|$.
\end{defn}

In other words, unraveling the definition of $\theta_\sigma$ in condition (2), we must have  $\phi(y) \geq \phi(z_i)$, where $i$ is the number appearing latest in the word of $\sigma$ such that $y \succeq z_i$.

\begin{example} \label{ex:tp}
Shown below are the $1423$-partition $\theta$ of a poset $P$ (left), and an example of a $(2314, 1423)$-labeling $\phi$ (right).  Here $z_1, z_2 , z_3, z_4$ are taken to be the minimal elements from left to right.

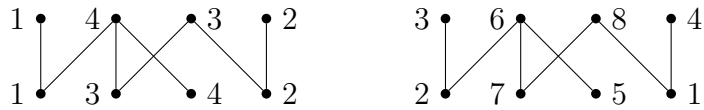
\begin{center}
\begin{tikzpicture}
 [auto,
 vertex/.style={circle,draw=black!100,fill=black!100,thick,inner sep=0pt,minimum size=1mm}]
\node (v1) at ( 0,0) [vertex, label=left:$1$] {};
\node (v2) at ( 1,0) [vertex, label=left:$3$] {};
\node (v3) at (2,0) [vertex, label=right:$4$] {};
\node (v4) at ( 3,0) [vertex, label=right:$2$] {};
\node (v5) at (0,1) [vertex, label=left:$1$] {};
\node (v6) at (1,1) [vertex, label=left:$4$] {};
\node (v7) at (2, 1) [vertex, label=right:$3$] {};
\node (v8) at (3, 1) [vertex, label=right:$2$] {};
\draw [-] (v1) to (v5);
\draw [-] (v1) to (v6);
\draw [-] (v2) to (v6);
\draw [-] (v2) to (v7);
\draw [-] (v3) to (v6);
\draw [-] (v4) to (v7);
\draw [-] (v4) to (v8);
\end{tikzpicture}
\hspace{1 cm}
\begin{tikzpicture}
 [auto,
 vertex/.style={circle,draw=black!100,fill=black!100,thick,inner sep=0pt,minimum size=1mm}]
\node (v1) at ( 0,0) [vertex, label=left:$2$] {};
\node (v2) at ( 1,0) [vertex, label=left:$7$] {};
\node (v3) at (2,0) [vertex, label=right:$5$] {};
\node (v4) at ( 3,0) [vertex, label=right:$1$] {};
\node (v5) at (0,1) [vertex, label=left:$3$] {};
\node (v6) at (1,1) [vertex, label=left:$6$] {};
\node (v7) at (2, 1) [vertex, label=right:$8$] {};
\node (v8) at (3, 1) [vertex, label=right:$4$] {};
\draw [-] (v1) to (v5);
\draw [-] (v1) to (v6);
\draw [-] (v2) to (v6);
\draw [-] (v2) to (v7);
\draw [-] (v3) to (v6);
\draw [-] (v4) to (v7);
\draw [-] (v4) to (v8);
\end{tikzpicture}
\end{center}

It is clear that condition $(1)$ holds since $\phi(z_2) > \phi(z_3) > \phi(z_1) > \phi(z_4)$.

For condition (2), consider for instance the element $x$ that covers $z_1$, $z_2$, and $z_3$.  Since $3$ appears last among $1$, $2$, and $3$ in $\sigma = 1423$, it follows that $\theta(x) = \theta(z_3) = \sigma^{-1}(3) = 4$, so we must have $\phi(x) \geq \phi(z_3)$.
%$z_2$ and $z_4$. Since $\theta(x) = \theta(z_2) = 3$, we must have $\phi(x) \geq \phi(z_2)$.
\end{example}

%We can count the number of $(\pi, \sigma)$-labelings by
%\[ T_P(\pi, \sigma) = \frac{n!}{\pi(\alpha(\sigma)_{\pi_m}\dots\alpha(\sigma)_{\pi_1})} .\] DOUBLE CHECK THIS!

Consider the operations $f_i , g_i\colon S_m \rightarrow S_m$ for $i=1, \dots, m-1$ defined by 
\begin{align*}
f_i(\pi) &= \pi_1\cdots\widehat{\pi_i}\cdots\pi_m\pi_i,\\
g_i(\pi) &= \pi_1\cdots\pi_{i-1}\pi_m\pi_i \cdots\pi_{m-1}.
\end{align*}
Observe that $f_i$ and $g_i$ are inverses of each other.  For any subset $S = \{s_1, \dots, s_l\} \subseteq [m-1]$ with $s_1 < \dots < s_l$, define 
\begin{align*}
f_S &= f_{s_1} \circ f_{s_2} \circ \cdots \circ f_{s_l},\\
g_S &= g_{s_l} \circ g_{s_{l-1}} \circ \cdots \circ g_{s_1}.
\end{align*}
Explicitly,
\[f_S(\pi) = \pi_1\cdots\widehat{\pi_{s_1}}\cdots\widehat{\pi_{s_2}}\cdots\widehat{\pi_{s_l}}\cdots\pi_m\pi_{s_l}\pi_{s_{l-1}}\cdots\pi_{s_1},\]
and $g_S$ is the inverse of $f_S$.

We can now give the following description of $\MaxOne(\psi_{\alpha(\sigma)})$.

\begin{lemma}\label{Expressing Omega Psi}
Let $P$ be a poset of size $n$ with $m$ minimal elements, and let $\sigma \in S_m$. Then
\[(n-1)!\MaxOne(\psi_{\alpha(\sigma)}) = \sum_{S \subseteq [m-1]} (-1)^{|S|}|T_P(g_S(\sigma), \sigma)|. \]
\end{lemma}

\begin{proof}
%Each set $S$ on the righthand side can have anywhere from $0$ to $m-1$ elements.  Therefore we can write the righthand side as 
%\[ \sum_{i=1}^{m}\sum_{\substack{S \subseteq [m-1] \\|S| = m-i}} (-1)^{|S|}|T_P(g_S(\sigma), \sigma)|. \]
%Let $|S| = m-i$. 
%For a fixed value of $i$, %$\sum_{\substack{S \subseteq [m-1] \\|S| = m-i}} |T_P(g_S(\sigma), \sigma)|$ counts the size of the following union: 
%\[\sum_{\substack{S \subseteq [m-1] \\|S| = m-i}} |T_P(g_S(\sigma), \sigma)|=\bigcup_{\substack{S \subseteq [m-1] \\|S| = m-i}}T_P(g_S(\sigma), \sigma)\]
%since these sets are disjoint.  
Note that for fixed $\sigma$, the permutations $g_S(\sigma)$ are distinct for distinct sets $S$, so the sets $T_P(g_S(\sigma), \sigma)$ are all disjoint.

Fix a value of $i$ with $1 \leq i \leq m$. Every labeling $\phi$ of $P$ in the union 
\[\bigcup_{\substack{S \subseteq [m-1] \\|S| = m-i}}T_P(g_S(\sigma), \sigma)\]
must satisfy the following two conditions:
\begin{enumerate}
\item $\phi(z_{\sigma_{1}}) > \phi(z_{\sigma_{2}}) > \dots > \phi(z_{\sigma_{i-1}}) > \phi(z_{\sigma_i})< \phi(z_{\sigma_{i+1}}) < \dots < \phi(z_{\sigma_{m-1}}) <  \phi(z_{\sigma_{m}})$; and
\item for all $y \in P$, $\phi(y) \geq \phi(z)$, where $z$ is the minimal element such that $\theta_\sigma(z) = \theta_\sigma(y)$.
%\phi(z_{\sigma_{\phi_\sigma(y)}})$where $\phi_\sigma$ is the $\sigma$-partition of $P$.
\end{enumerate}
Equivalently, for all $y \in P$, $\phi(y)$ must be at least $\phi(z_{\sigma_j})$ for any $j$ between $i$ and $\theta_\sigma(y)$, inclusive. It follows that (1) and (2) are equivalent to the following:
\begin{itemize}
\item $\phi(z_{\sigma_i})=1$ (which of course is the smallest label among all $n$ elements of $P$);
\item for $j = i-1, i-2, \dots, 1$: $z_{\sigma_j}$ has the smallest label among %$\V({\{\sigma_{z_1}, \dots, \sigma_{z_{i-1}}\}}) \setminus \V({\{\sigma_{z_{j+1}}, \dots, \sigma_{z_{i-1}}\}})$, and
$P \setminus \V(z_{\sigma_{j+1}}, \dots, z_{\sigma_m})$, and
\item for $j = i+1, i+2, \dots, m$: $z_{\sigma_j}$ has the smallest label among $\V(z_{\sigma_{j}}, \dots, z_{\sigma_{m}})$.
\end{itemize}

%We will now count the number of such labelings by finding the probability that a random labeling satisfies the previous three conditions.  The probability that in a random labeling of $P$, the label of $z_{\sigma_i}$ is the smallest of all $n$ elements is $\frac{1}{n}$.  The probability that the label of $z_{\sigma_j}$ with $j<i$, is the smallest out of $\V({\{\sigma_{z_1}, \dots, \sigma_{z_{i-1}}\}}) \setminus \V({\{\sigma_{z_{j+1}}, \dots, \sigma_{z_{i-1}}\}})$ is $\frac{1}{|\V({\{\sigma_{z_1}, \dots, \sigma_{z_{i-1}}\}}) \setminus \V({\{\sigma_{z_{j+1}}, \dots, \sigma_{z_{i-1}}\}})|}$.  The probability that the label of $z_{\sigma_j}$ with $j>i$, is the smallest out of $\V({\{\sigma_{z_{j}}, \dots, \sigma_{z_{m}}\}})$ is $\frac{1}{|\V({\{\sigma_{z_{j}}, \dots, \sigma_{z_{m}}\}})|}$.  
%These probabilities are independent of each other.
It is easy to check that for a random bijection from $P$ to $[n]$, these events are all independent of one another. If we let $\alpha = \alpha(\sigma)$, then
\begin{align*}
|P \setminus V(z_{\sigma_{j+1}}, \dots, z_{\sigma_m})| &= \alpha_1 + \cdots + \alpha_j,\\
|V(z_{\sigma_j}, \dots, z_{\sigma_m})| &= \alpha_j + \cdots + \alpha_m.
\end{align*}
Hence the number of such labelings $\phi$ is
\[\frac{n!}{n \cdot \prod_{j=1}^{i-1}(\alpha_1 + \cdots + \alpha_j) \cdot \prod_{j=i+1}^m (\alpha_j + \cdots + \alpha_m)} = \frac{(n-1)!}{\pi(\alpha_1\cdots\alpha_{i-1}) \cdot \pi((\alpha_{i+1}\cdots\alpha_{m})^{\textrm{rev}})}.\]

%The product of these probabilities is
%\[\frac{1}{\pi(\alpha(\sigma)_1\cdots\alpha(\sigma)_{i-1})\cdot n \cdot \pi((\alpha(\sigma)_{i+1}\cdots\alpha(\sigma)_{m})^r)}.\]
%Therefore the number of such labelings is 
%\[\frac{n!}{\pi(\alpha(\sigma)_1\cdots\alpha(\sigma)_{i-1})\cdot n \cdot \pi((\alpha(\sigma)_{i+1}\cdots\alpha(\sigma)_{m})^r)}.\]
When we include the $(-1)^{|S|}$ and sum over all values of $i$, we find that the right hand side equals
\[\sum_{i=1}^{m} \frac{(-1)^{m-i}(n-1)!}{\pi(\alpha_1\cdots\alpha_{i-1}) \cdot \pi((\alpha_{i+1}\cdots\alpha_{m})^{\textrm{rev}})}.\]
This value is exactly $(n-1)!\MaxOne(\psi_{\alpha(\sigma)})$ by Lemma~\ref{MaxOne expansion}, as desired.
\end{proof}

The following two lemmas about $T_P(\pi, \sigma)$ will be helpful for giving a combinatorial interpretation for $\MaxOne(\KPmin)$. This is because many of the terms on the right hand side of Lemma \ref{Expressing Omega Psi} for different $\sigma$ will cancel out due to the principle of inclusion-exclusion.

%\color{red}\textbf{ $\phi(y) \geq \phi(z)$ where $z$ is the min element s.t. $\phi_\sigma(z) = \phi_\sigma(y)$.  Maybe change the letter phi to something else}\color{black}
\begin{lemma}\label{set intersection}
Let $P$ be a poset with $m$ minimal elements.  For all $S \subseteq [m-1]$ and all $\pi \in S_m$,
\[T_P(\pi, f_S(\pi)) = \bigcap_{s \in S}T_P(\pi, f_s(\pi)).\]
\end{lemma}

\begin{proof}
%Suppose that $y \in P$.
We will show that the two sets above are equivalent by showing that the inequalities that $\phi$ must satisfy are the same in both sets.  Clearly, a labeling $\phi$ must satisfy $\phi(z_{\pi_1}) > \phi(z_{\pi_2}) > \cdots > \phi(z_{\pi_m})$ to be in either of the sets.  Since this determines the inequalities that must be satisfied among minimal elements of $P$, we will turn our attention to some non-minimal element $y \in P$.

Define $S_y = \{s \in [m] \mid y \succ z_{\pi_{s}}\}$. Then the condition on $\phi(y)$ for $\phi \in T_P(\pi, \sigma)$ is
$\phi(y) \geq \phi(z_{\pi_s})$, where $\pi_s$ is the value that, over all $s \in S_y$, appears latest in the word of $\sigma$.
%determined by which number $\pi_s$ (for $s \in S_y$) appears last in $\sigma$. 
Let $a = \max(S_y) = \theta_\pi(y)$ (where $\theta_\pi$ is the $\pi$-partition of $P$), and let $s' = \min(S_y \cap S)$ if $S_y \cap S \neq \varnothing$.
We claim that any labeling $\phi$ in either side of our desired equality must satisfy $\phi(y) \geq \phi(z_{\pi_{s'}})$ if $S_y \cap S \neq \varnothing$, and $\phi(y) \geq \phi(z_{\pi_{a}})$ otherwise.

%, and if $S_y$ is nonempty, let $s' = \min(S_y)$.  Let $a = \theta_\pi(y)$, where $\theta_\pi$ is the $\pi$-partition of $P$.  Since $a$ is by definition the largest number such that $y \succ z_{\pi_a}$, it follows that every element of $S_y$ is at most $a$.
%%if $y \succ z_{\pi_j}$ for some $j$, then $j \leq a$.
%We will now show that any labeling $\phi$ in either the left or right hand side of our desired equality must satisfy $\phi(y) \geq \phi(z_{\pi_{s'}})$ if $S_y$ is nonempty, and $\phi(y) \geq \phi(z_{\pi_{a}})$ otherwise.

For the left hand side, $f_S(\pi)$ by definition reverses the order of $\pi_s$ for $s \in S_y \cap S$ and places these elements after any element of $S_y \setminus S$. Thus the last $\pi_s$ for $s \in S_y$ appearing in $f_S(\pi)$ will be $\pi_{s'}$ if $S_y \cap S \neq \varnothing$ and $\pi_a$ otherwise, as desired.

For the right hand side, a labeling $\phi$ in $T_P(\pi, f_s(\pi))$ must satisfy $\phi(y) \geq \phi(z_{\pi_s})$ if $s\in S_y \cap S$ or $\phi(y) \geq \phi(z_{\pi_a})$ otherwise.  Since we have that $\phi(z_{\pi_1}) > \phi(z_{\pi_2}) > \cdots > \phi(z_{\pi_m})$ and $s' = \min(S_y \cap S)$, it follows that the strictest inequality that $\phi$ must satisfy on the right hand side is $\phi(y) \geq \phi(z_{\pi_{s'}})$ if $S_y \cap S$ is nonempty (when $s = s'$), and $\phi(y) \geq \phi(z_{\pi_{a}})$ otherwise, as desired.
\end{proof}

\begin{example}
Consider again the poset shown in Example~\ref{ex:tp}, and let $\pi = 2314$, so that $\phi(z_2) > \phi(z_3) > \phi(z_1) > \phi(z_4)$.
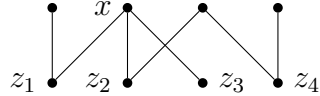
\[\begin{tikzpicture}
 [auto,
 vertex/.style={circle,draw=black!100,fill=black!100,thick,inner sep=0pt,minimum size=1mm}]
\node (v1) at ( 0,0) [vertex, label=left:$z_1$] {};
\node (v2) at ( 1,0) [vertex, label=left:$z_2$] {};
\node (v3) at (2,0) [vertex, label=right:$z_3$] {};
\node (v4) at ( 3,0) [vertex, label=right:$z_4$] {};
\node (v5) at (0,1) [vertex] {};
\node (v6) at (1,1) [vertex, label=left:$x$] {};
\node (v7) at (2, 1) [vertex] {};
\node (v8) at (3, 1) [vertex] {};
\draw [-] (v1) to (v5);
\draw [-] (v1) to (v6);
\draw [-] (v2) to (v6);
\draw [-] (v2) to (v7);
\draw [-] (v3) to (v6);
\draw [-] (v4) to (v7);
\draw [-] (v4) to (v8);
\end{tikzpicture}\]
%Since $a$ and $d$ only cover one element each, any labeling $\phi$ in $T_P(\pi, \sigma)$ for any $\sigma$ must satisfy $\phi(a) \geq \phi(z_1)$ and $\phi(d) \geq \phi(z_4)$.
Since the element $x$ covers $z_1$, $z_2$, and $z_3$, the condition on $\phi(x)$ in $T_P(\pi, \sigma)$ depends on which of $1$, $2$, and $3$ appears latest in $\sigma$. 
For $s = 1, 2, 3$, we then have the following values of $f_s(\pi)$ with the corresponding conditions on $\phi(x)$ for each $\phi$ in $T_P(\pi, f_s(\pi))$.
\begin{alignat*}{2}
f_1(\pi) &= \mathbf{31}4\mathbf{2}\colon\qquad & \phi(x) & \geq  \phi(z_2),\\
f_2(\pi) &= \mathbf{21}4\mathbf{3}\colon\qquad & \phi(x) & \geq  \phi(z_3),\\
f_3(\pi) &= \mathbf{23}4\mathbf{1}\colon\qquad & \phi(x) & \geq  \phi(z_1).
\end{alignat*}
Suppose that $S = \{1, 2\}$. Then in $\bigcap_{s \in S}T_P(\pi, f_s(\pi)) = T_P(\pi, f_1(\pi)) \cap T_P(\pi, f_2(\pi))$, we must have $\phi(x) \geq \max\{\phi(z_2), \phi(z_3)\} = \phi(z_2)$. But this is also the condition imposed on $\phi(x)$ in $T_P(\pi, f_S(\pi))$ since $f_S(\pi) = \textbf{1}4\textbf{32}$.
\end{example}

We will now see that for all $i$, any $(\pi, f_i(\pi))$-labeling is also a $(\pi, \pi)$-labeling.

\begin{lemma}\label{set union}
Let $P$ be a poset with $m$ minimal elements. For all $\pi \in S_m$,
\[\bigcup_{i=1}^{m-1}T_P(\pi, f_i(\pi)) \subseteq T_P(\pi, \pi).\]
\end{lemma}

\begin{proof}
Let $\phi \in \bigcup_{i=1}^{m-1}T_P(\pi, f_i(\pi)).$ This means that $\phi$ is a $(\pi, f_i(\pi))$-labeling for some value of $i$.  No matter the value of $i$, $\phi$ must satisfy
$\phi(z_{\pi_1}) > \phi(z_{\pi_2}) > \dots > \phi(z_{\pi_m})$, which also must be satisfied to be in $T_P(\pi, \pi)$.

It remains to show that for all $y \in P$, $\phi(y) \geq \phi(z_{\pi_a})$, where $a = \theta_\pi(y)$, and $\theta_\pi$ is the $\pi$-partition of $P$.  If $y \nsucceq z_{\pi_i}$, then $\phi(y) \geq \phi(z_{\pi_{a}})$.  If instead $y \succeq z_{\pi_i}$, then $\phi(y) \geq \phi(z_{\pi_i})$, but $\phi(z_{\pi_i}) \geq \phi(z_{\pi_{a}})$ since $i \leq a$ by the definition of $\theta_\pi$.  Therefore $\phi \in T_P(\pi, \pi)$.
%
%It remains to show that for all $y \in P$, $\phi(y) \geq \phi(z_{\pi_{\theta_\pi(y)}})$ where $\theta_\pi$ is the $\pi$-partition of $P$.  We will consider the case when $y \nsucceq z_{\pi_i}$ and when $y \succeq z_{\pi_i}$. If $y \nsucceq z_{\pi_i}$, then $\phi(y) \geq \phi(z_{\pi_{\theta_\pi(y)}})$.  If $y \succeq z_{\pi_i}$, then $\phi(y) \geq \phi(z_{\pi_i})$, but $\phi(z_{\pi_i}) \geq \phi(z_{\pi_{\theta_\pi(y)}})$.  Therefore $\phi \in T_P(\pi, \pi)$ and $\cup_{i=1}^{m-1}T_P(\pi, f_i(\pi)) \subseteq T_P(\pi, \pi)$.
\end{proof}

As we will show in Lemma \ref{Sum nonnegative}, a $(\pi, \pi)$-labeling $\phi$ appears on the right hand side of Lemma~\ref{set union} but not the left hand side if and only if, for all $i = 1, \dots, m-1$, there exists $y \succ z_{\pi_i}$ with $\phi(y) < \phi(z_{\pi_i})$. We will describe such labelings more thoroughly in the next subsection.

\subsection{Zigzag labelings}

In this subsection, we introduce the notion of a zigzag labeling, which will allow us to give a combinatorial interpretation for $\MaxOne(\KPmin)$.

%The left hand side in Lemma \ref{set union} is a proper subset of the right hand side whenever there exists a $(\pi, \pi)$-labeling $\phi$ such that if $x$ is minimal and $\phi(x) \neq 1$, then there exists a $y \succ x$ with $\phi(y) < \phi(x)$.  We will see that the structure of the poset determines if such a labeling exists.  

\begin{defn}
A bijection $\phi\colon P \to [n]$ is a \emph{zigzag labeling} if the following conditions hold:
\begin{enumerate}
\item if $x$ is minimal and $\phi(x) \neq 1$, then there exists $y \succ x$ with $\phi(y) < \phi(x)$; and
\item if $x$ is not minimal, then there exists a minimal element $y \prec x$ with $\phi(y) < \phi(x)$.
\end{enumerate}
We say that a zigzag labeling $\phi$ has \emph{type} $\pi$ if $\phi(z_{\pi_1}) > \phi(z_{\pi_2}) > \dots > \phi(z_{\pi_m})$.
\end{defn}

\begin{example}
%The following is an example of a $(4231, 4231)$-labeling of $P$ that is not in $\cup_{i=1}^{3}t_P(4231, f_i(4231)) = t_P(4231, 2314) \cup t_P(4231,4312) \cup t_P(4231,4213)$.  Again we take $z_1, z_2 , z_3, z_4$ to be drawn from left to right.
The following is an example of a zigzag labeling of $P$ of type $4231$.  (We take $z_1, z_2 , z_3, z_4$ to be the minimal elements drawn from left to right.)
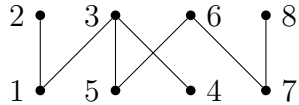
\[
\begin{tikzpicture}
 [auto,
 vertex/.style={circle,draw=black!100,fill=black!100,thick,inner sep=0pt,minimum size=1mm}]
\node (v1) at ( 0,0) [vertex, label=left:$1$] {};
\node (v2) at ( 1,0) [vertex, label=left:$5$] {};
\node (v3) at (2,0) [vertex, label=right:$4$] {};
\node (v4) at ( 3,0) [vertex, label=right:$7$] {};
\node (v5) at (0,1) [vertex, label=left:$2$] {};
\node (v6) at (1,1) [vertex, label=left:$3$] {};
\node (v7) at (2, 1) [vertex, label=right:$6$] {};
\node (v8) at (3, 1) [vertex, label=right:$8$] {};
\draw [-] (v1) to (v5);
\draw [-] (v1) to (v6);
\draw [-] (v2) to (v6);
\draw [-] (v2) to (v7);
\draw [-] (v3) to (v6);
\draw [-] (v4) to (v7);
\draw [-] (v4) to (v8);
\end{tikzpicture}
\]
\end{example}

%It should be noted that the zigzag labelings of type $\pi$ are exactly the elements of  \[T_P(\pi, \pi) \setminus \textstyle\bigcup_{i=1}^{m-1}T_P(\pi, f_i(\pi)).\]  Therefore the total number of zigzag labelings of any type is counted by 
%\[\sum_{\pi \in S_m} \left| T_P(\pi, \pi) \setminus \textstyle\bigcup_{i=1}^{m-1}T_P(\pi, f_i(\pi)) \right|.\]
%We say that $\phi$ is a \emph{zigzag labeling of type $\pi$} if $\phi \in t_P(\pi, \pi) - \bigcup_{i=1}^{m-1}t_P(\pi, f_i(\pi))$.
%\textbf{Put this in a definition block and move example 5.12 to before 5.10}

Zigzag labelings can be used to determine whether or not a poset is connected.

\begin{lemma}\label{Connected have zigzag label}
	A poset $P$ has a zigzag labeling if and only if it is connected.
%If $P$ is connected, then $P$ has a zigzag labeling.
%there exists a permutation $\pi \in S_m$ where $P$ has a zigzag labeling of type $\pi$.  
%If $P$ is disconnected, then no such permutation exists.
\end{lemma}

\begin{proof}
Suppose $P$ is connected and has $m$ minimal elements.  Since $P$ is connected, there exists a permutation $\pi' \in S_m$ such that $\V(z_{\pi'_1}, \dots, z_{\pi'_{i-1}}) \cap \V(z_{\pi'_i}) \neq \varnothing$ for all $i$.
Indeed, if it were not possible to find $\pi'_i$ given $\pi'_1, \dots, \pi'_{i-1}$, then $\V(z_{\pi'_1}, \dots, z_{\pi'_{i-1}})$ would be a connected component of $P$, which contradicts the fact that $P$ is connected.

An example of a zigzag labeling $\phi$ of $P$ is as follows:  label $z_{\pi'_1}$ with $1$, and label the rest of $\V(z_{\pi'_1})$ with $2, 3, \dots$.  Then label $z_{\pi'_2}$ with the lowest number available, and label the elements of $\V(z_{\pi'_1}, z_{\pi'_2}) \setminus \V(z_{\pi'_1})$ with the next lowest remaining labels.  Continue this process until all elements of $P$ are labeled. For all $i > 1$, there exists an element $x \in \V(z_{\pi'_{1}}, \dots, z_{\pi'_{i-1}}) \cap \V(z_{\pi'_i})$, and this element satisfies $x \succ z_{\pi'_i}$ and $\phi(x) < \phi(z_{\pi'_i})$.  Therefore $\phi$ is a zigzag labeling. (One can check that $\phi$ has type $\pi=(\pi')^{\textrm{rev}}$.)

%$\phi \notin \cup_{j=1}^{m-1}T_P(\pi, f_j(\pi))$.  Therefore 

%Let $\pi$ be the reverse of $\pi'$, that is $\pi_i = \pi'_{m+1-i}$.
%Note that $\phi$ is a $(\pi, \pi)$-labeling.
%Since for all $i < m$, there exists an $x \in \V({\{z_{\pi_i}\}}) \cap \V({\{z_{\pi_{i+1}}, \dots, z_{\pi_{m}}\}}) $ such that $\phi(z_{\pi_i}) > \phi(x)$ (in fact all $x$ in the intersection satisfy that inequality), %$\phi \notin \cup_{j=1}^{m-1}T_P(\pi, f_j(\pi))$.  Therefore $\phi$ is a zigzag labeling of type $\pi$.

Now suppose $P$ is disconnected, and let $\phi\colon P \to [n]$ be any bijection. Let $x$ be the element with the smallest label that is not in the same connected component as the element labeled $1$. Then $x$ is not comparable to any element $y$ with a smaller label, so $\phi$ cannot be a zigzag labeling.
%
%Now suppose $P$ is disconnected.  This means for every permutation $\pi \in S_m$, there exists an $i \neq m$ such that $P = \V({\{z_{\pi_1}, \dots, z_{\pi_{i}}\}}) \sqcup \V({\{z_{\pi_{i+1}}, \dots, z_{\pi_{m}}\}})$.  If a zigzag labeling $\phi$ of type $\pi$ existed we would have $\phi(z_{\pi_1}) > \phi(z_{\pi_2}) > \dots > \phi(z_{\pi_m})$.  In particular, $\phi(z_{\pi_i}) \neq 1$ and $\phi(z_{\pi_1}) > \phi(z_{\pi_2}) > \dots > \phi(z_{\pi_i})$. 
%
%Since $\V({\{z_{\pi_i}\}}) \cap \V({\{z_{\pi_{i+1}}, \dots, z_{\pi_{m}}\}}) = \varnothing$ and for all $x \in \V({\{z_{\pi_1}, \dots, z_{\pi_i}\}})$ $\phi(x) > \phi(z_{\pi_i})$, there exists no $y \succ z_{\pi_i}$ with $\phi(y) < \phi(z_{\pi_i})$.  Therefore there are no zigzag labelings of $P$ when it is disconnected.
\end{proof}

We will now use the lemmas in the previous subsection to show that we can count zigzag labelings of a fixed type by an alternating sum.

\begin{lemma}\label{Sum nonnegative}
Let $P$ be a finite poset with $m$ minimal elements, and let $\pi \in S_m$. The set of zigzag labelings of $P$ of type $\pi$ is
\[T_P(\pi, \pi) \setminus \bigcup_{i=1}^{m-1}T_P(\pi, f_i(\pi)),\]
and the number of such labelings is
\[\sum_{S \subseteq [m-1]} (-1)^{|S|}|T_P(\pi, f_S(\pi))|.\]
%counts the number of zigzag labelings of $P$ of type $\pi$.
%\[\sum_{S \subseteq [m-1]} (-1)^{|S|}|T_P(\pi, f_s(\pi))| \geq 0\]
%and if $P$ is connected, there exists a $\pi \in S_m$ where 
%\[\sum_{S \subseteq [m-1]} (-1)^{|S|}|T_P(\pi, f_s(\pi))| > 0.\]
\end{lemma}

\begin{proof}
	First note that for a labeling $\phi$ of type $\pi$, condition (2) for being a $(\pi, \pi)$-labeling is equivalent to condition (2) for being a zigzag labeling. Indeed, for any non-minimal $x \in P$, if there exists a minimal element $z_{\pi_i} \prec x$ with $\phi(z_{\pi_i}) < \phi(x)$, then certainly we can take $i$ to be the maximum value such that $z_{\pi_i} \prec x$ since that will only decrease the value of $\phi(z_{\pi_i})$.
	
	A labeling $\phi \in T_P(\pi, \pi)$ does not lie in $T_P(\pi, f_i(\pi))$ if there exists some element $y \succ z_{\pi_i}$ with $\phi(y) < \phi(z_{\pi_i})$. Hence this occurs for all $i < m$ when $\phi$ satisfies condition (1) of being a zigzag labeling, which completes the proof of the first claim.
	
	From the principle of inclusion-exclusion and Lemma \ref{set intersection},
\[ \sum_{S \subseteq [m-1]} (-1)^{|S|}|T_P(\pi, f_S(\pi))| = |T_P(\pi, \pi)| - \left|\bigcup_{i=1}^{m-1}T_P(\pi, f_i(\pi))\right|.\]
By Lemma \ref{set union}, we have that 
\[|T_P(\pi, \pi)| - \left|\bigcup_{i=1}^{m-1}T_P(\pi, f_i(\pi))\right| = \left|T_P(\pi, \pi) \setminus \bigcup_{i=1}^{m-1}T_P(\pi, f_i(\pi))\right|.\]
This is exactly the number of zigzag labelings of type $\pi$.
%Observe that $|T_P(\pi, \pi)| - |\bigcup_{i=1}^{m-1}T_P(\pi, f_i(\pi))|$ counts the number of zigzag labelings of type $\pi$.  When $P$ is connected, Lemma \ref{Connected have zigzag label} states that exists a permutation $\pi$ where this number is positive.
\end{proof}

%The proof of Lemma \ref{Sum nonnegative} shows that the number of zigzag labelings of $P$ with type $\pi$ is counted by $\sum_{S \subseteq [m-1]} (-1)^{|S|}|T_P(\pi, f_s(\pi))|$.  %It also shows that if $P$ is connected, then there exists a permutation $\pi$ such that $P$ has a zigzag labeling of type $\pi$.

%\begin{cor}\label{Big Cor}
%Let $P$ be a finite poset and $\pi \in S_m$, then
%\[\sum_{S \subseteq [m-1]} (-1)^{|S|}|T_P(\pi, f_s(\pi))| \geq 0\]
%and if $P$ is connected, there exists a $\pi \in S_m$ where 
%\[\sum_{S \subseteq [m-1]} (-1)^{|S|}|T_P(\pi, f_s(\pi))| > 0.\]
%\end{cor}
%
%\begin{proof}
%By Lemma \ref{Sum nonnegative}, we have that the alternating sum is nonnegative since it counts the number of zigzag labelings.  By Lemma \ref{Connected have zigzag label}, there exists a $\pi \in S_m$ that makes the alternating sum positive when $P$ is connected.
%\end{proof}

We now have the required tools to prove that the number of zigzag labelings of $P$ can be determined from $\KP$.

%\textbf{This thm should say $(n-1)!\MaxOne(\KPmin)$ counts the number of zigzag labelings}
\begin{thm}\label{Total zigzags}
The number of zigzag labelings of $P$ is counted by $(n-1)!\MaxOne(\KPmin)$.
%If $P$ is a finite, naturally labeled poset, then
%\[(n-1)!\MaxOne(\KPmin) = \sum_{\sigma \in S_m} \sum_{S \subseteq [m-1]} (-1)^{|S|}T_P(g_S(\sigma), \sigma) =  \sum_{\pi \in S_m} \sum_{S \subseteq [m-1]} (-1)^{|S|}T_P(\pi, f_s(\pi)).\]
\end{thm}

\begin{proof}
By Lemmas \ref{KPmin Expansion} and \ref{Expressing Omega Psi}, %it follows that
\begin{align*}
(n-1)!\MaxOne(\KPmin) &= (n-1)!\sum_{\sigma \in S_m}\MaxOne(\psi_{\alpha(\sigma)})\\
&= \sum_{\sigma \in S_m} \sum_{S \subseteq [m-1]} (-1)^{|S|}|T_P(g_S(\sigma), \sigma)|
\end{align*}
%But $\MaxOne$ is a linear function so we can rewrite the expression as 
%\[ (n-1)!\MaxOne(\sum_{\sigma \in S_m}\psi_{\alpha(\sigma)}) = \sum_{\sigma \in S_m}(n-1)!\MaxOne(\psi_{\alpha(\sigma)}).\]
%Therefore by Lemma \ref{Expressing Omega Psi},
%\[ (n-1)!\MaxOne(\KPmin) = \sum_{\sigma \in S_m} \sum_{S \subseteq [m-1]} (-1)^{|S|}|T_P(g_S(\sigma), \sigma)|.\]
For all $\pi \in S_m$, we can group together the terms for which $g_S(\sigma) = \pi$, meaning $\sigma = f_S(\pi)$ since $g_S$ and $f_S$ are inverses.  It follows that 
\[ \sum_{\sigma \in S_m} \sum_{S \subseteq [m-1]} (-1)^{|S|}|T_P(g_S(\sigma), \sigma)| =  \sum_{\pi \in S_m} \sum_{S \subseteq [m-1]} (-1)^{|S|}|T_P(\pi, f_S(\pi))|. \]
By Lemma \ref{Sum nonnegative}, this counts the number of zigzag labelings of $P$ (of any type).
\end{proof}

\begin{example}
Let $P$ be the following naturally labeled poset. 
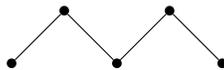
\[
\begin{tikzpicture}
 [auto,
 vertex/.style={circle,draw=black!100,fill=black!100,thick,inner sep=0pt,minimum size=1mm}, scale = .7]
\node (v1) at ( 0,0) [vertex] {};
\node (v2) at ( 1,1) [vertex] {};
\node (v3) at (2,0) [vertex] {};
\node (v4) at ( 3,1) [vertex] {};
\node (v5) at ( 4,0) [vertex] {};
\draw [-] (v1) to (v2);
\draw [-] (v2) to (v3);
\draw [-] (v3) to (v4);
\draw [-] (v4) to (v5);
\end{tikzpicture}
\]
One can compute
\[\KP = 4\psi_{122} + 2\psi_{113}  + 4\psi_{1211}  + 8\psi_{1121} + 4\psi_{1112} + 16\psi_{11111}.\]
It follows that 
\[\KPmin = 4\psi_{122} + 2\psi_{113}.\] 
Applying $\MaxOne$ and multiplying by $4!$, we have \[4!\cdot\MaxOne(\KPmin) = 24(-\tfrac{1}{6} + \tfrac{1}{2}) =8.\]
The following are the $8$ zigzag labelings of $P$:
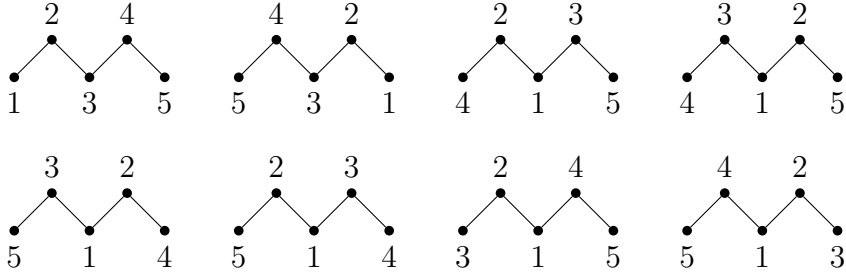
\[
\begin{tikzpicture}
 [auto,
 vertex/.style={circle,draw=black!100,fill=black!100,thick,inner sep=0pt,minimum size=1mm}, scale = .5]
\node (v1) at ( 0,0) [vertex,label=below:$1$] {};
\node (v2) at ( 1,1) [vertex,label=above:$2$] {};
\node (v3) at (2,0) [vertex,label=below:$3$] {};
\node (v4) at ( 3,1) [vertex,label=above:$4$] {};
\node (v5) at ( 4,0) [vertex,label=below:$5$] {};
\draw [-] (v1) to (v2);
\draw [-] (v2) to (v3);
\draw [-] (v3) to (v4);
\draw [-] (v4) to (v5);
\end{tikzpicture}
%%%%%%%
\hspace{.5 cm}
\begin{tikzpicture}
 [auto,
 vertex/.style={circle,draw=black!100,fill=black!100,thick,inner sep=0pt,minimum size=1mm}, scale = .5]
\node (v1) at ( 0,0) [vertex,label=below:$5$] {};
\node (v2) at ( 1,1) [vertex,label=above:$4$] {};
\node (v3) at (2,0) [vertex,label=below:$3$] {};
\node (v4) at ( 3,1) [vertex,label=above:$2$] {};
\node (v5) at ( 4,0) [vertex,label=below:$1$] {};
\draw [-] (v1) to (v2);
\draw [-] (v2) to (v3);
\draw [-] (v3) to (v4);
\draw [-] (v4) to (v5);
\end{tikzpicture}
%%%%%%%
\hspace{.5 cm}
\begin{tikzpicture}
 [auto,
 vertex/.style={circle,draw=black!100,fill=black!100,thick,inner sep=0pt,minimum size=1mm}, scale = .5]
\node (v1) at ( 0,0) [vertex,label=below:$4$] {};
\node (v2) at ( 1,1) [vertex,label=above:$2$] {};
\node (v3) at (2,0) [vertex,label=below:$1$] {};
\node (v4) at ( 3,1) [vertex,label=above:$3$] {};
\node (v5) at ( 4,0) [vertex,label=below:$5$] {};
\draw [-] (v1) to (v2);
\draw [-] (v2) to (v3);
\draw [-] (v3) to (v4);
\draw [-] (v4) to (v5);
\end{tikzpicture}
%%%%%%%
\hspace{.5 cm}
\begin{tikzpicture}
 [auto,
 vertex/.style={circle,draw=black!100,fill=black!100,thick,inner sep=0pt,minimum size=1mm}, scale = .5]
\node (v1) at ( 0,0) [vertex,label=below:$4$] {};
\node (v2) at ( 1,1) [vertex,label=above:$3$] {};
\node (v3) at (2,0) [vertex,label=below:$1$] {};
\node (v4) at ( 3,1) [vertex,label=above:$2$] {};
\node (v5) at ( 4,0) [vertex,label=below:$5$] {};
\draw [-] (v1) to (v2);
\draw [-] (v2) to (v3);
\draw [-] (v3) to (v4);
\draw [-] (v4) to (v5);
\end{tikzpicture}
\]
\[
\begin{tikzpicture}
 [auto,
 vertex/.style={circle,draw=black!100,fill=black!100,thick,inner sep=0pt,minimum size=1mm}, scale = .5]
\node (v1) at ( 0,0) [vertex,label=below:$5$] {};
\node (v2) at ( 1,1) [vertex,label=above:$3$] {};
\node (v3) at (2,0) [vertex,label=below:$1$] {};
\node (v4) at ( 3,1) [vertex,label=above:$2$] {};
\node (v5) at ( 4,0) [vertex,label=below:$4$] {};
\draw [-] (v1) to (v2);
\draw [-] (v2) to (v3);
\draw [-] (v3) to (v4);
\draw [-] (v4) to (v5);
\end{tikzpicture}
\hspace{.5 cm}
\begin{tikzpicture}
 [auto,
 vertex/.style={circle,draw=black!100,fill=black!100,thick,inner sep=0pt,minimum size=1mm}, scale = .5]
\node (v1) at ( 0,0) [vertex,label=below:$5$] {};
\node (v2) at ( 1,1) [vertex,label=above:$2$] {};
\node (v3) at (2,0) [vertex,label=below:$1$] {};
\node (v4) at ( 3,1) [vertex,label=above:$3$] {};
\node (v5) at ( 4,0) [vertex,label=below:$4$] {};
\draw [-] (v1) to (v2);
\draw [-] (v2) to (v3);
\draw [-] (v3) to (v4);
\draw [-] (v4) to (v5);
\end{tikzpicture}
\hspace{.5 cm}
\begin{tikzpicture}
 [auto,
 vertex/.style={circle,draw=black!100,fill=black!100,thick,inner sep=0pt,minimum size=1mm}, scale = .5]
\node (v1) at ( 0,0) [vertex,label=below:$3$] {};
\node (v2) at ( 1,1) [vertex,label=above:$2$] {};
\node (v3) at (2,0) [vertex,label=below:$1$] {};
\node (v4) at ( 3,1) [vertex,label=above:$4$] {};
\node (v5) at ( 4,0) [vertex,label=below:$5$] {};
\draw [-] (v1) to (v2);
\draw [-] (v2) to (v3);
\draw [-] (v3) to (v4);
\draw [-] (v4) to (v5);
\end{tikzpicture}
\hspace{.5 cm}
\begin{tikzpicture}
 [auto,
 vertex/.style={circle,draw=black!100,fill=black!100,thick,inner sep=0pt,minimum size=1mm}, scale = .5]
\node (v1) at ( 0,0) [vertex,label=below:$5$] {};
\node (v2) at ( 1,1) [vertex,label=above:$4$] {};
\node (v3) at (2,0) [vertex,label=below:$1$] {};
\node (v4) at ( 3,1) [vertex,label=above:$2$] {};
\node (v5) at ( 4,0) [vertex,label=below:$3$] {};
\draw [-] (v1) to (v2);
\draw [-] (v2) to (v3);
\draw [-] (v3) to (v4);
\draw [-] (v4) to (v5);
\end{tikzpicture}
\]
\end{example}

%The total number of zigzag labelings of $P$ of any type is counted by $(n-1)!\MaxOne(\KPmin)$.

Now that we have this combinatorial interpretation, it is easy to show our main result about the irreduciblity of $\KP$.

\begin{thm}\label{Conn and Irr}
A poset $P$ is connected if and only if $\KP$ is irreducible over $\QSym$.
\end{thm}
\begin{proof}
	The ``if" direction is true by Remark \ref{Disconnect posets}.  The ``only if" direction is an immediate consequence of Theorem \ref{Total zigzags} and Lemma \ref{Connected have zigzag label}. If $P$ is connected, then $(n-1)! \MaxOne(\KPmin) > 0$, so it follows from Lemma~\ref{Max One product zero} that $\KPmin$ is irreducible. Therefore by Lemma \ref{Reducible KP}, $\KP$ is irreducible.
%	This is an immediate consequence of Theorem \ref{Total zigzags} and Lemma \ref{Connected have zigzag label}. If $P$ is connected, then $(n-1)! \MaxOne(\KPmin) > 0$, so it follows from Lemma~\ref{Max One product zero} that $\KPmin$ is irreducible. Therefore by Lemma \ref{Reducible KP}, $\KP$ is irreducible.
\end{proof}

In fact, this method implies that if $P$ is connected, then $\KP$ does not even lie in the span of the homogeneous, reducible elements of $\QSym$.

%\begin{cor}
%There exists a zigzag labeling of $P$ if and only if $P$ is connected. 
%\end{cor}
This result also tells us how $\KP$ factors as a product of irreducible partition generating functions. It is well known that $\QSym$ is isomorphic to a polynomial ring and hence is a unique factorization domain (see \cite{Hazewinkel, LamPylyavskyy}).

\begin{cor}\label{Factoring KP}
	Let $P$ be a naturally labeled poset. Then the irreducible factorization of $\KP$ is given by $\KP = \prod_i K_{P_i}(\x)$, where $P_i$ are the connected components of $P$.
\end{cor}

\begin{proof}
This follows from the fact that $\KP$ factors into a product of the partition generation functions of its connected components, and each of these is irreducible by Theorem \ref{Conn and Irr}. 
\end{proof}

%We say the \emph{connected decomposition} of $P$ is $P = P_1 \sqcup P_2 \sqcup \dots \sqcup P_k$ where each $P_i$ is connected.  Every poset has a unique connected decomposition up to reordering.
%We write $P = P_1 \sqcup P_2 \sqcup \dots \sqcup P_k$ where $P_1, \dots, P_k$ are the connected components of $P$.  
%Let $P_1, \dots, P_k$ be the connected components of $P$.  We say that $P$ is the \emph{disjoint union} of $P_1, \dots, P_k$, written $P = P_1 \sqcup P_2 \sqcup \dots \sqcup P_k$.

This result also gives a condition on when two posets can have the same partition generating function based on their connected components.

\begin{cor}
	Let $P$ and $Q$ be naturally labeled posets. Let $P_1, \dots, P_k$ be the connected components of $P$, and $Q_1, \dots, Q_l$ the connected components of $Q$.
%$P = P_1 \sqcup P_2 \sqcup \dots \sqcup P_k$ and $Q = Q_1 \sqcup Q_2 \sqcup \dots \sqcup Q_l$.
If $\KP = \KQ$, then $k = l$, and there exists a permutation $\pi \in S_k$ such that $K_{P_i}(\mathbf{x}) = K_{Q_{\pi_{i}}}(\mathbf{x})$ for all $i$.
\end{cor}

\begin{proof}
This follows immediately from Corollary \ref{Factoring KP} and the fact that QSym is a unique factorization domain. (There are no scalar factors since Theorem~\ref{L Expansion} implies that the coefficient of $L_n$ in the expansion of any $K_{P_i}(\x)$ or $K_{Q_j}(\x)$ in the $L_\alpha$-basis is $1$.)
\end{proof}

It is still open whether a connected poset $(P, \omega)$ that is not naturally labeled need have an irreducible generating function $\KPomega$. Our approach will not work in this more general setting: for a general labeled poset $\Pomega$, it is possible for $\KPomega$ to be irreducible even when $\widetilde{K}_{(P, \omega)}(\mathbf{x})$ is reducible. (Moreover, $\KPomega$ may be irreducible yet lie in the span of homogeneous, reducible elements of $\QSym$.)

\begin{example}
Let $(P, \omega)$ be the the following labeled poset with the generating function shown.  (This can either be computed from the fundamental basis expansion using Theorem~\ref{L Expansion} and the results of \cite{BallantineEtc}, or from Theorem \ref{Combinatorial Description of Psi} below.)
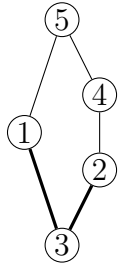
\[
\begin{tikzpicture}
 [auto,
 vertex/.style={circle,draw=black!100,fill=black!100,thick,inner sep=0pt,minimum size=1mm}]
%\node (v1) at ( 0,1.5) [vertex, label=left:$1$] {};
%\node (v1) at ( 0,1.5) [vertex] {};
\node (v1) [draw,circle,minimum size=.45cm,inner sep=0pt] at (0,1.5) {$1$};
%\node (v2) at ( 1,1) [vertex] {};
\node (v2) [draw,circle,minimum size=.45cm,inner sep=0pt] at (1,1) {$2$};
%\node (v3) at (.5,0) [vertex] {};
\node (v3) [draw,circle,minimum size=.45cm,inner sep=0pt] at (.5,0) {$3$};
%\node (v4) at ( 1,2) [vertex] {};
\node (v4) [draw,circle,minimum size=.45cm,inner sep=0pt] at (1,2) {$4$};
%\node (v5) at (.5, 3) [vertex] {};
\node (v5) [draw,circle,minimum size=.45cm,inner sep=0pt] at (.5,3) {$5$};
\draw [very thick] [-] (v3) to (v1);
\draw [very thick] [-] (v3) to (v2);
\draw [-] (v2) to (v4);
\draw [-] (v1) to (v5);
\draw [-] (v4) to (v5);
\end{tikzpicture}
\]
\begin{align*}
K_{(P, \omega)}(\mathbf{x}) = &-\psi_{32} -\psi_{23} - \psi_{311} - \psi_{221} - 3\psi_{212} + \psi_{122} \\
&+\psi_{113} - 3\psi_{2111} + \psi_{1211} + \psi_{1121} + 3\psi_{1112} + 3\psi_{11111}.
\end{align*}

The minimum length part of $K_{(P, \omega)}$ is then %$\widetilde{K}_{(P, \omega)}(\mathbf{x}) = \tau_{2}(K_{(P, \omega)}(\mathbf{x}))$, then we have
\begin{align*}
\widetilde{K}_{(P, \omega)}(\mathbf{x}) &= -\psi_{32} -\psi_{23} \\
&=-\psi_3\cdot\psi_2.
\end{align*}

Since $\widetilde{K}_{(P, \omega)}(\mathbf{x})$ is reducible, it follows that $\MaxOne(\widetilde{K}_{(P, \omega)}(\mathbf{x})) = 0$ even though $P$ is connected. However, one can check that $K_{(P, \omega)}$ is itself irreducible (for instance, by considering the coefficients of $\psi_{311}$, $\psi_{131}$, and $\psi_{113}$).
\end{example}

\section{Series-parallel posets}
We will now turn our attention to a collection of naturally labeled posets known as \emph{series-parallel posets}. We will use the results of the previous section to show that distinct series-parallel posets have distinct partition generating functions.

\begin{defn}
	The class $\mathcal{SP}$ of \emph{series-parallel posets} is the smallest collection of posets that satisfies the following:
	\begin{itemize}
		\item the one-element poset $1_\mathcal{P}$ lies in $\mathcal{SP}$;
		\item if $P \in \mathcal{SP}$ and $Q \in \mathcal{SP}$, then $P \sqcup Q \in \mathcal{SP}$; and
		\item if $P \in \mathcal{SP}$ and $Q \in \mathcal{SP}$, then $P \oplus Q \in \mathcal{SP}$. 
	\end{itemize}
\end{defn}
Here, the \emph{disjoint union} $P \sqcup Q$ is the poset on the disjoint union of $P$ and $Q$ with relations $x\preceq y$ if and only if $x \preceq_P y$ or $x \preceq_Q y$, while the \emph{ordinal sum} $P \oplus Q$ is the poset on the disjoint union of $P$ and $Q$ with relations $x\preceq y$ if and only if $x \preceq_P y$, $x \preceq_Q y$, or $x \in P$ and $y \in Q$.  

It is well known that $P \in \mathcal{SP}$ if and only if $P$ is $N$-free \cite{Stanley1}. (A poset is \emph{$N$-free} if there does not exist an induced poset on four elements $\{a, b, c, d\} \subseteq P$ with relations $a \prec b \succ c \prec d$.)

%We say that a poset $P$ is \emph{series parallel} if $P \in \mathcal{SP}$.  
Note that if a series-parallel poset is disconnected, then it can be expressed as the disjoint union of series-parallel posets.  If it is connected (and has more than one element), then it can be expressed as an ordinal sum of series-parallel posets.

In \cite{HasebeTsujie}, Hasebe and Tsujie show that if a poset $P$ is a finite rooted tree, then it can be distinguished by its $P$-partition generating function.  They then ask if the same can be said about series-parallel posets.  We will now give an answer to this question.

\begin{thm}\label{Series Parallel}  Let $P$ be a series-parallel poset. Then $P$ is uniquely determined (up to isomorphism) by $\KP$.
\end{thm}

\begin{proof}
We will prove this by induction on the size of $P$.  If $P$ has a single element, then the result holds trivially since there is a unique single-element poset.  Now suppose that the results holds for all series-parallel posets with fewer than $|P|$ elements.

If $P$ is disconnected, then $P = P_1 \sqcup P_2 \sqcup \dots \sqcup P_k$, where each $P_i$ is a connected series-parallel poset.  By Lemma \ref{Factoring KP}, the irreducible factors of $\KP$ are the partition generating functions of the connected components of $P$.  By induction, we can determine the connected components from their partition generating functions.

We will now assume that $P$ is connected and series-parallel.  We can express $P$ as $P = P_1 \oplus P_2 \oplus \dots \oplus P_k$, where each $P_i$ is series-parallel.  By Lemma $2.7$ in \cite{LiuWeselcouch}, we can determine $K_{P_i}(\mathbf{x})$ from $\KP$ for all $i$.  Since each $P_i$ is series-parallel, it follows by induction that it is uniquely determined by $K_{P_i}(\mathbf{x})$.%Since $P_i$ is series parallel, $P_i = P_{i,1} \sqcup P_{i,2} \sqcup \dots \sqcup P_{i,l}$ where $P_{i,j}$ is a connected series parallel poset.  By Lemma \ref{Factoring KP}, we can determine $K_{P_{i,j}}(\mathbf{x})$ for all values of $i$ and $j$.  By induction, $P_{i,j}$ is determined by $K_{P_{i,j}}(\mathbf{x})$.
\end{proof}

It should be noted that Theorem \ref{Series Parallel} does not hold when $\Pomega$ is not naturally labeled.

%\color{red} Make all "series parallel" "series-parallel" with a dash \color{black}

\begin{example}\label{series parallel example}
The following labeled series-parallel posets have the same partition generating functions.
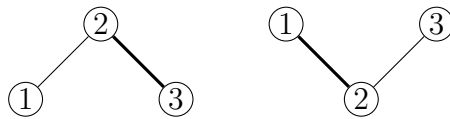
\[
\begin{tikzpicture}
 [auto,
 vertex/.style={circle,draw=black!100,fill=black!100,thick,inner sep=0pt,minimum size=1mm}]
%\node (v1) at ( 0,0) [vertex] {};
\node (v1) [draw,circle,minimum size=.45cm,inner sep=0pt] at (0,0) {$1$};
%\node (v2) at ( 1,1) [vertex] {};
\node (v2) [draw,circle,minimum size=.45cm,inner sep=0pt] at (1,1) {$2$};
%\node (v3) at (2,0) [vertex] {};
\node (v3) [draw,circle,minimum size=.45cm,inner sep=0pt] at (2,0) {$3$};
\draw [-] (v1) to (v2);
\draw [very thick] [-] (v2) to (v3);
\end{tikzpicture}
\hspace{1 cm}
\begin{tikzpicture}
 [auto,
 vertex/.style={circle,draw=black!100,fill=black!100,thick,inner sep=0pt,minimum size=1mm}]
%\node (v1) at ( 0,1) [vertex] {};
\node (v1) [draw,circle,minimum size=.45cm,inner sep=0pt] at (0,1) {$1$};
%\node (v2) at ( 1,0) [vertex] {};
\node (v2) [draw,circle,minimum size=.45cm,inner sep=0pt] at (1,0) {$2$};
%\node (v3) at (2,1) [vertex] {};
\node (v3) [draw,circle,minimum size=.45cm,inner sep=0pt] at (2,1) {$3$};
\draw [very thick] [-] (v1) to (v2);
\draw  [-] (v2) to (v3);
\end{tikzpicture}
\]
\end{example}

Although the labeled posets in Example \ref{series parallel example} are not isomorphic, they are both series-parallel.
%This leads us to the following question.

\begin{quest}
Let $(P, \omega_1)$ and $(Q, \omega_2)$ be labeled posets, and suppose that $P$ is series-parallel.  If $K_{(P, \omega_1)}(\x) = K_{(Q, \omega_2)}(\x)$, does this imply that $Q$ is series-parallel?
\end{quest}

\section{A Combinatorial Description of the Coefficients}
In this section, we will give a combinatorial interpretation for the coefficients in the $\psi_{\alpha}$-expansion of $\KPomega$ akin to the Murnaghan-Nakayama rule.

\subsection{Generalized border-strips}

We will begin by considering the following question: for which labeled posets $(P, \omega)$ does $\MinOne(K_{(P, \omega)}(\x)) \neq 0$? By Lemma~\ref{MaxOne expansion}, this is equivalent to asking when the composition $(n)$ lies in the $\psi$-support of $\KPomega$.

As mentioned earlier, in the fundamental quasisymmetric function basis $\{L_\alpha\}$,
\[
  \MinOne(L_{\alpha}) =
        \begin{cases}
            (-1)^{k} & \text{if $\alpha = (1^k, n-k)$,} \\
            0 & \text{otherwise.}
        \end{cases}
\]
This means that we must only consider linear extensions of $\Pomega$ whose descent set is $\{1, \dots, k\}$ for some value of $k$ (or empty when $k=0$) when computing $\MinOne(\KPomega)$.

\begin{lemma}\label{Avoids natural-strict}
Suppose $(P, \omega)$ is a labeled poset.  If there exists a chain $a \prec b \prec c$ such that $\omega(a) < \omega(b) > \omega(c)$, then $\MinOne(\KPomega) = 0$.
\end{lemma}

\begin{proof}
Every linear extension of $\Pomega$ has an ascent at some point between the appearance of $\omega(a)$ and $\omega(b)$, as well as a descent between the appearance of $\omega(b)$ and $\omega(c)$.  Therefore $\KPomega$ has no compositions of the form $\alpha = (1^k, n-k)$ in its $L$-support, so $\MinOne(\KPomega) = 0$.
\end{proof}

We say that a labeled poset $\Pomega$ is a \emph{generalized border-strip} if $\Pomega$ does not contain a chain $a \prec b \prec c$ such that $\omega(a) < \omega(b) > \omega(c)$.  Note that any generalized border-strip has a maximum order ideal $I$ containing only strict edges, in which case $P \setminus I$ is naturally labeled (i.e., contains only natural edges).  The use of the term ``generalized border-strip" will become clear in Section \ref{Murnaghan-Nakayama rule section} when we consider labeled posets that arise from semistandard Young tableaux.

Denote by $J$ the maximal subset of $I$ such that $(P \setminus I) \cup J$ is naturally labeled and $I \setminus J$ is an order ideal. In other words, $J$ is the set of elements $j$ such that the principal order ideal generated by $j$ contains only strict edges, while the principal filter generated by $j$ contains only natural edges. Then $J$ must be a subset of the maximal elements of $I$.  Note that for all $j \in J$ and $x \in P$, if $j$ and $x$ share an edge in the Hasse diagram of $(P, \omega)$, then $\omega(x) > \omega(j)$.  Observe that $J$ is nonempty since the element labeled $1$ (which we denote by $1$) must be a maximal element in $I$, and $(P \setminus I) \cup \{1\}$ is naturally labeled, so $1 \in J$.
\begin{example}
Suppose that $\Pomega$ is the following generalized border-strip.
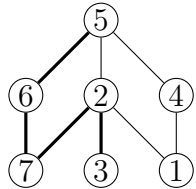
\[
\begin{tikzpicture}
 [auto,
 vertex/.style={circle,draw=black!100,fill=black!100,thick,inner sep=0pt,minimum size=1mm}]
%\node (v1) at ( 0, 0) [vertex, label=left:$7$] {};
%\node (v2) at ( 1,0) [vertex, label=left:$3$] {};
%\node (v3) at (2,0)  [vertex, label=left:$1$] {};
%\node (v4) at ( 0,1) [vertex, label=left:$6$] {};
%\node (v5) at (1, 1) [vertex, label=left:$2$] {};
%\node (v6) at (2, 1) [vertex, label=left:$4$] {};
%\node (v7) at (1, 2) [vertex, label=left:$5$] {};
\node (v1) [draw,circle,minimum size=.45cm,inner sep=0pt] at (0,0) {$7$};
\node (v2) [draw,circle,minimum size=.45cm,inner sep=0pt] at (1,0) {$3$};
\node (v3) [draw,circle,minimum size=.45cm,inner sep=0pt] at (2,0) {$1$};
\node (v4) [draw,circle,minimum size=.45cm,inner sep=0pt] at (0,1) {$6$};
\node (v5) [draw,circle,minimum size=.45cm,inner sep=0pt] at (1,1) {$2$};
\node (v6) [draw,circle,minimum size=.45cm,inner sep=0pt] at (2,1) {$4$};
\node (v7) [draw,circle,minimum size=.45cm,inner sep=0pt] at (1,2) {$5$};
\draw [very thick] [-] (v1) to (v4);
\draw [very thick] [-] (v1) to (v5);
\draw [very thick] [-] (v2) to (v5);
\draw [very thick] [-] (v4) to (v7);
\draw [-] (v3) to (v5);
\draw [-] (v3) to (v6);
\draw [-] (v5) to (v7);
\draw [-] (v6) to (v7);
\end{tikzpicture}
\]
The ideal $I$ has maximal elements $\{1, 3, 6 \}$, while $J = \{1\}$.
\end{example}

This set $J$ can be used to compute the value of $\MinOne(\KPomega)$ for generalized border-strips.

\begin{lemma}\label{Computing MinOne}
Suppose $\Pomega$ is a generalized border-strip, and let $I$ and $J$ be defined as above.  Then 
\[
  \MinOne(\KPomega) =
        \begin{cases}
            (-1)^{|I|-1} & \text{if $J = \{1\}$,} \\
            0 & \text{otherwise.}
        \end{cases}
\]
\end{lemma}

\begin{proof}
We only need to consider the linear extensions of $\Pomega$ whose descent set is $\{1, \dots, k\}$ for some value of $k$.  
%Let $A$ be any subset of $J \setminus \{1\}$.  
The only linear extensions of $\Pomega$ with a descent set of this form must begin with, for some subset $A \subseteq J \setminus \{1\}$, the elements of $I \setminus A$ in decreasing order and end with the elements of $(P \setminus I ) \cup A$ in increasing order.
%The only linear extensions of $\Pomega$ with this descent set begin with the entries of $I \setminus A$ in decreasing order and end with $(P \setminus I ) \cup A$ in increasing order where $A$ is any subset of $J \setminus \{1\}$.
Therefore we can express $\MinOne(\KPomega)$ as 
\[ \MinOne(\KPomega) = \sum_{A \subseteq J \setminus \{1\}} (-1)^{|I|-|A|-1}.\]
This value is $0$ unless $J = \{1\}$, in which case $\MinOne(\KPomega) = (-1)^{|I|-1}$.
%Therefore 
%\[ \MinOne(\KPomega)  = (-1)^{|I|-1}\delta_{|J|, 1},\]
%as desired.
\end{proof}

Observe that if $P$ is naturally labeled, then $\MinOne(\KP) =1$ if and only if $P$ has a unique minimal element.  (In this case, $I$ and $J$ are both just the set of minimal elements of $P$.)

We say that $\Pomega$ is \emph{rooted} if $\MinOne(\KPomega) \neq 0$. In other words, $\Pomega$ is rooted if it is a generalized border-strip and $J=\{1\}$, that is, $1$ is the unique element whose principal ideal contains only strict edges and whose principal filter contains only natural edges.
\begin{cor} \label{cor:connected}
	If $\Pomega$ is rooted, then it is connected.
\end{cor}

\begin{proof}
	If $\Pomega$ is disconnected, then $\KPomega$ is the product of the $\Pomega$-partition generating functions of the connected components of $P$.  By Lemma \ref{Max One product zero}, $\MinOne(\KPomega) = 0$.
	
	Alternatively, the elements with the minimum label in each connected component of $P$ lie in the set $J$.
\end{proof}

\subsection{Enriched $\Pomega$-partitions}
We will now give an alternate combinatorial description of rooted posets in terms of enriched $\Pomega$-partitions. We will then use this to give a combinatorial formula (with signs) for the coefficient of $\psi_\alpha$ in $\KPomega$.

 %We say that a labeled poset $(P, \omega)$ is a \emph{generalized rim hook} if $\MinOne(\KPomega) \neq 0$.  Our choice of name will become clear when we apply our results to skew Schur functions.
Let $\mathbf{P'}$ be the nonzero integers with the total order
\[ -1<_{\mathbf{P'}} 1 <_{\mathbf{P'}} -2 <_{\mathbf{P'}} 2 <_{\mathbf{P'}} \cdots . \]
For $k \in \mathbf{P'}$, we define the absolute value of $k$ by $|k| = i \in \Z^+$ if $k \in \{-i , i\}$, and we say $k < 0$ if $k = -|k|$, and $k >0$ otherwise.

The following definition is due to Stembridge \cite{Stembridge}.
\begin{defn} %Stembridge (in the email Molly sent will citations)
Let $\Pomega$ be a labeled poset.  An \emph{enriched $\Pomega$-partition} is a map $f \colon P \rightarrow \mathbf{P'}$ such that $|f|$ is surjective onto $[k]$ for some $k$, and if $x$ is covered by $y$, then we have
\begin{enumerate}[(i)]
\item $f(x) \leq_{\mathbf{P'}} f(y)$,
\item if $|f(x)| = |f(y)|$ and $\omega(x) < \omega(y)$, then $f(y) > 0$,
\item if $|f(x)| = |f(y)|$ and $\omega(x) > \omega(y)$, then $f(x) < 0$.
\end{enumerate}
\end{defn}

It should be noted that if $x \prec y \prec z$ is a chain with $\omega(x) < \omega(y) > \omega(z)$, then for all enriched $\Pomega$-partitions $f$, $|f(x)| < |f(z)|$.  It follows that for each $i$, the subposet
$P_i = \{x \in P \mid |f(x)| = i\}$ is a generalized border-strip.

Suppose that $f$ is an enriched $\Pomega$-partition and we know the value of $|f(x)|$ for all $x$. Is this enough information to determine the value of $f(x)$ for all $x$?  The answer to this question is no.

\begin{example}\label{Enriched}
Let $\Pomega$ be the following poset.  Suppose $f$ is an enriched $\Pomega$-partition, and $|f(x)| = 1$ for all $x$.
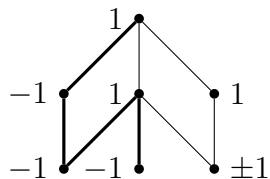
\[
\begin{tikzpicture}
 [auto,
 vertex/.style={circle,draw=black!100,fill=black!100,thick,inner sep=0pt,minimum size=1mm}]
\node (v1) at ( 0, 0) [vertex, label=left:$-1$] {};
\node (v2) at ( 1,0) [vertex, label=left:$-1$] {};
\node (v3) at (2,0) [vertex, label=right:$\pm 1$] {};
\node (v4) at ( 0,1) [vertex, label=left:$-1$] {};
\node (v5) at (1, 1) [vertex, label=left:$1$] {};
\node (v6) at (2, 1) [vertex, label=right:$1$] {};
\node (v7) at (1, 2) [vertex, label=left:$1$] {};
\draw [very thick] [-] (v1) to (v4);
\draw [very thick] [-] (v1) to (v5);
\draw [very thick] [-] (v2) to (v5);
\draw [very thick] [-] (v4) to (v7);
\draw [-] (v3) to (v5);
\draw [-] (v3) to (v6);
\draw [-] (v5) to (v7);
\draw [-] (v6) to (v7);
\end{tikzpicture}
\]
The element in the bottom right can either be sent to $-1$ or $1$ by $f$.
%Since there is $1$ unlabeled vertex, $\Pomega$ is a generalized rim hook.  It follows that $\MinOne(\KPomega)  = (-1)^{3}$.
\end{example}

For each $x\in P$ and enriched $\Pomega$-partition $f$, define the map $f_x \colon P \rightarrow \mathbf{P'}$ by 
\[
  f_x(y) =
        \begin{cases}
            -f(x) & \text{if $ x = y$,} \\
            f(y) & \text{otherwise.}
        \end{cases}
\]
We say that an element $x \in P$ is \emph{ambiguous with respect to $f$} if $f_x$ is still an enriched $\Pomega$-partition.  %Denote by $\AMB(f)$ the set of elements that are ambiguous with respect to $f$.  
%Note that determining whether $x$ is ambiguous depends only the generalized ribbon consisting of all elements labeled $\pm f(x)$. Indeed, the ambiguous elements in this generalized ribbon exactly correspond to the elements of the set $J$ described in the previous section.

Let $\mathbf{P^*} = \mathbf{P'} \cup \{1^*, 2^*, \dots \}$, where each $i^*$ satisfies  $-i <_{\mathbf{P^*}} i^* <_{\mathbf{P^*}} i$ and $|i^*| = i$.  For an enriched $\Pomega$-partition $f$, consider the map $f^* \colon P \rightarrow \mathbf{P^*}$ defined by
\[
f^*(x) = 
	\begin{cases}
	|f(x)|^* & \text{if $x$ is ambiguous with respect to $f$,} \\
	f(x) & \text{otherwise.}
	\end{cases}
\]
The map $f^*$ is called a \emph{starred $\Pomega$-partition}. Note that $f^*$ is uniquely determined by the generalized border-strips $P_i$.

%We say that $f \sim g$ if $\AMB(f) = \AMB(g)$, $|f(x)| = |g(x)|$ for all $x$, and $f(x) \neq g(x)$ implies $x \in \AMB(f)$.  The relation $\sim$ is an equivalence relation.
%The \emph{ambiguity of $f^*$} is the sequence $\amb(f) = (a_1, a_2, \dots, a_k)$ where $a_i$ is number of $x \in P$ such that $x \in \AMB(f)$ and $|f(x)| = i$.  Equivalently, $a_i = |(f^*)^{-1}(i^*)|$, the number of elements sent to $i^*$ by $f^*$.

The \emph{ambiguity of $f^*$} is the sequence $\amb(f^*) = (a_1, a_2, \dots, a_k)$, where $a_i = |(f^*)^{-1}(i^*)|$, the number of elements labeled $i^*$ by $f^*$,
The \emph{sign of $f^*$} is
\[\sign(f^*) = (-1)^{|\{x \mid f^*(x) < 0\}|},\]
and the \emph{weight of $f^*$} is 
%\[\comp(f^*) = ((f^*)^{-1}(-1)+(f^*)^{-1}(1) + (f^*)^{-1}(1^*) , \dots, (f^*)^{-1}(-k)+(f^*)^{-1}(k)) + (f^*)^{-1}(k^*).\]
$\comp(f^*) = (b_1, b_2, \dots, b_k)$,
where $b_i$ is the total number of elements labeled $-i$, $i^*$, or $i$ by $f$.
%\[b_i = |(f^*)^{-1}(-i)|+|(f^*)^{-1}(i)| + |(f^*)^{-1}(i^*)|\] for all $i$.
%Observe that if $f \sim g$, then $\amb(f) = \amb(g)$, $\sign(f) = \sign(g)$ and $\comp(f) = \comp(g)$  Therefore it makes sense to use the notation $\amb([f])$, $\sign([f])$, and $\comp([f])$ where $[f]$ is the equivalence class of $f$.
%In Example \ref{Enriched}, $\amb(f^*) = (1)$, $\sign(f^*) = (-1)^3$, and $\comp(f^*) = (7)$.

\begin{example}%\label{Enriched}
Let $f$ be the enriched $\Pomega$-partition from Example \ref{Enriched}. Then the starred $\Pomega$-partition $f^*$ is shown below.
\[
\begin{tikzpicture}
 [auto,
 vertex/.style={circle,draw=black!100,fill=black!100,thick,inner sep=0pt,minimum size=1mm}]
\node (v1) at ( 0, 0) [vertex, label=left:$-1$] {};
\node (v2) at ( 1,0) [vertex, label=left:$-1$] {};
\node (v3) at (2,0) [vertex, label=right:$1^*$] {};
\node (v4) at ( 0,1) [vertex, label=left:$-1$] {};
\node (v5) at (1, 1) [vertex, label=left:$1$] {};
\node (v6) at (2, 1) [vertex, label=right:$1$] {};
\node (v7) at (1, 2) [vertex, label=left:$1$] {};
\draw [very thick] [-] (v1) to (v4);
\draw [very thick] [-] (v1) to (v5);
\draw [very thick] [-] (v2) to (v5);
\draw [very thick] [-] (v4) to (v7);
\draw [-] (v3) to (v5);
\draw [-] (v3) to (v6);
\draw [-] (v5) to (v7);
\draw [-] (v6) to (v7);
\end{tikzpicture}
\]
Here, $\amb(f^*) = (1)$, $\sign(f^*) = (-1)^3$, and $\comp(f^*) = (7)$.
\end{example}

%\begin{lemma}  Let $f$ be an enriched $\Pomega$-partition such that $|f|$ is surjective onto $[k]$.  If $\amb(f^*) = (a_1, a_2, \dots, a_k)$, then $a_i \geq 1$ for all $i$.
%\end{lemma}

%\begin{proof}

%If $\amb(f^*) = (a_1, a_2, \dots, a_k)$, then $a_i \geq 1$ for all $i$. This is because each $P_i$ is a generalized ribbon, and the ambiguous elements are the elements in $J$ described in the previous section, which we already showed to be nonempty.
%Let $P_i$ be the following subposet of $P$: $P_i = \{x \in P \mid |f(x)| = i\}$.  Each $P_i$ is a generalized ribbon.  The ambiguous elements are the elements in $J$ described earlier which we already showed was nonempty.
%Suppose that $y$ is the minimally labeled element of $P_i$.  We will show that $y$ is ambiguous with respect to $f$.  Note that if $y$ shares an edge with an element $x \in P$, then either $|f(y)| < |f(x)|$, $|f(y)| = |f(x)|$, or $|f(y)| > |f(x)|$.  We only need to consider the case where $|f(y)| = |f(x)|$ because the other two cases are trivial.  Suppose that $|f(y)| = |f(x)|$.  Since $y$ is minimally labeled in $P_i$, it follows that $\omega(y) < \omega(x)$.  
%If $y$ is covered by $x$, then replacing $f(y)$ with $-f(y)$ does not lead to a contradiction since condition (ii) is a condition about $f(x)$.  Similarly, if $y$ covers $x$, then replacing $f(y)$ with $-f(y)$ does not lead to a contradiction because condition (iii) is a condition about $f(x)$.
%\end{proof}

We say that a starred $\Pomega$-partition $f^*$ is a \emph{pointed $\Pomega$-partition} if $\amb(f^*) = (1, 1, \dots, 1)$.  
%Observe that any linear extension of $\Pomega$ gives rise to a unique, pointed $\Pomega$-partition.  
For instance, if $f \colon P \rightarrow [n]$ is any linear extension of $\Pomega$, then $f^*$ is a pointed $\Pomega$-partition. When $P$ is naturally labeled, $f^*$ will be pointed whenever each subposet $P_i = \{x \in P \mid |f^*(x)| = i\}$ has a unique minimal element, so this agrees with our previous definition of pointed $P$-partitions.

The following lemma is equivalent to Lemma~\ref{Computing MinOne}.
\begin{lemma}\label{Combinatorial Description of MinOne}Let $\Pomega$ be a labeled poset. Then there exists a pointed $\Pomega$-partition $f^*$ with $\comp(f^*) = (n)$ if and only if $\Pomega$ is rooted. In this case, $\MinOne(\KPomega) = \sign(f^*)$; otherwise $\MinOne(\KPomega) = 0$.
\end{lemma}

\begin{proof}
We know from Lemma \ref{Avoids natural-strict} that in order for $\MinOne(\KPomega) \neq 0$, there cannot be a chain $a \prec b \prec c$ with $\omega(a) < \omega(b) > \omega(c)$.  If there are no such chains, then there exists a starred $\Pomega$-partition $f^*$ with $\comp(f^*) = (n)$. Explicitly, $f^*$ labels the elements of $P$ is as follows: if an element lies at the bottom of a strict edge, then it gets sent to $-1$; if it lies at the top of a natural edge, then it gets sent to $1$; otherwise, it gets sent to $1^*$.  There are no elements that lie both at the bottom of a strict edge and at the top of a natural edge.

The elements that get sent to $1^*$ are what we called $J$ in Lemma \ref{Computing MinOne}, and the elements that get sent to $-1$ or $1^*$ are what we called $I$.  Therefore the result holds by Lemma \ref{Computing MinOne}.
\end{proof}

%Let $\Pomega$ be a generalized rim hook.  We will relabel each entry of $P$ with either a $\bar{1}$ or $1$ in the following way: if $a$ is covered by $b$ and $\omega(a) > \omega(b)$, then label $a$ with $\bar{1}$, otherwise if $a$ is covered by $b$ and $\omega(a) < \omega(b)$, then label $b$ with $1$.

%Observe that since $\Pomega$ is a generalized ribbon no element will be labeled both $\bar{1}$ and $1$.  Some of the elements of $P$ will not be labeled, these elements are what we called $J$ previously, but since $\Pomega$ is a generalized rim hook there is a unique unlabeled element.  Note that $I$ is the set of elements that are either unlabeled or are labeled $\bar{1}$.  Therefore $\MinOne(\KPomega)  = (-1)^{|I|-1}$.

A pointed $(P, \omega)$-partition $f^*$ can be interpreted as a partitioning of $\Pomega$ such that each part is rooted. Specifically, each subposet $P_i = \{x \in P \mid |f^*(x)| = i\}$ is a rooted poset, and the weight of $f^*$ is the composition $(|P_1|, |P_2|, \dots)$. This allows us to give a combinatorial description of the coefficient of $\psi_\alpha$ in $\KPomega$.

\begin{thm}\label{Combinatorial Description of Psi}
Let $\Pomega$ be a labeled poset. Then
\[\KPomega = \sum_{f^*}\sign(f^*)\psi_{\comp(f^*)},\]
where the sum ranges over all pointed $\Pomega$-partitions $f^*$.
\end{thm}

\begin{proof}
Recall that Theorem \ref{coeff of psi alpha} states that the coefficient of $\psi_\alpha$ in $\KPomega$ can be computed by $\MinOne^{\otimes l}(\Delta_{\alpha}\KPomega)$, and that
\[ \Delta_{\alpha}\KPomega = \sum K_{(P_1, \omega)}(\x) \otimes \dots \otimes K_{(P_{l}, \omega)}(\x),\]
where $|P_i| = \alpha_i$, $P_1 \cup \dots \cup P_i$ is an order ideal of $P$, and $P_1, P_2, \dots, P_{l(\alpha)}$ partition $P$. 

Therefore the coefficient of $\psi_\alpha$ in $\KPomega$ is given by 
\[ \MinOne^{\otimes l}(\Delta_{\alpha}\KPomega) = \sum \MinOne(K_{(P_1, \omega)}(\x)) \otimes \dots \otimes \MinOne(K_{(P_l, \omega)}(\x)).\]
By Lemma \ref{Combinatorial Description of MinOne}, a term in this summation is $0$ unless each $(P_i, \omega)$ is rooted, meaning there exists a pointed $(P_i, \omega)$-partition $f_i^*$ with weight $(\alpha_i)$.  These can be combined to form a unique pointed $\Pomega$-partition $f^*$ with $\comp(f^*) = \alpha$.  Since $\MinOne(K_{(P_i, \omega)}(\x)) = \sign(f_i^*)$, % (-1)^{f_i^*}$, 
it follows that $\MinOne(K_{P_1}(\x)) \otimes \dots \otimes \MinOne(K_{P_l}(\x)) = \sign(f^*)$. %(-1)^{f^*}$.
\end{proof}

\begin{example}
Let $\Pomega$ be the following poset.
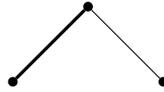
\[
\begin{tikzpicture}
 [auto,
 vertex/.style={circle,draw=black!100,fill=black!100,thick,inner sep=0pt,minimum size=1mm}]
\node (v1) at ( 0, 0) [vertex] {};
\node (v2) at ( 1,1) [vertex] {};
\node (v3) at (2,0) [vertex] {};
\draw [very thick] [-] (v1) to (v2);
\draw  [-] (v3) to (v2);
\end{tikzpicture}
\]
The following are the pointed $\Pomega$-partitions.
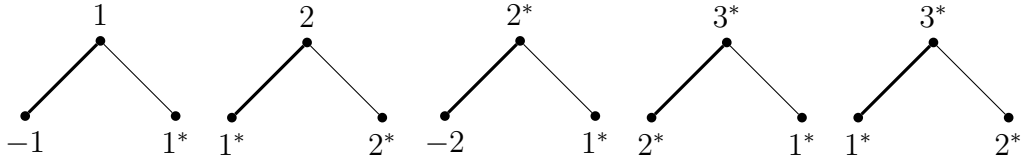
\[
\begin{tikzpicture}
 [auto,
 vertex/.style={circle,draw=black!100,fill=black!100,thick,inner sep=0pt,minimum size=1mm}]
\node (v1) at ( 0, 0) [vertex, label=below:$-1$] {};
\node (v2) at ( 1,1) [vertex, label=above:$1$] {};
\node (v3) at (2,0) [vertex, label=below:$1^*$] {};
\draw [very thick] [-] (v1) to (v2);
\draw  [-] (v3) to (v2);
\end{tikzpicture}
\hspace{.1 cm}
\begin{tikzpicture}
 [auto,
 vertex/.style={circle,draw=black!100,fill=black!100,thick,inner sep=0pt,minimum size=1mm}]
\node (v1) at ( 0, 0) [vertex, label=below:$1^*$] {};
\node (v2) at ( 1,1) [vertex, label=above:$2$] {};
\node (v3) at (2,0) [vertex, label=below:$2^*$] {};
\draw [very thick] [-] (v1) to (v2);
\draw  [-] (v3) to (v2);
\end{tikzpicture}
\hspace{.1 cm}
\begin{tikzpicture}
 [auto,
 vertex/.style={circle,draw=black!100,fill=black!100,thick,inner sep=0pt,minimum size=1mm}]
\node (v1) at ( 0, 0) [vertex, label=below:$-2$] {};
\node (v2) at ( 1,1) [vertex, label=above:$2^*$] {};
\node (v3) at (2,0) [vertex, label=below:$1^*$] {};
\draw [very thick] [-] (v1) to (v2);
\draw  [-] (v3) to (v2);
\end{tikzpicture}
\hspace{.1 cm}
\begin{tikzpicture}
 [auto,
 vertex/.style={circle,draw=black!100,fill=black!100,thick,inner sep=0pt,minimum size=1mm}]
\node (v1) at ( 0, 0) [vertex, label=below:$2^*$] {};
\node (v2) at ( 1,1) [vertex, label=above:$3^*$] {};
\node (v3) at (2,0) [vertex, label=below:$1^*$] {};
\draw [very thick] [-] (v1) to (v2);
\draw  [-] (v3) to (v2);
\end{tikzpicture}
\hspace{.1 cm}
\begin{tikzpicture}
 [auto,
 vertex/.style={circle,draw=black!100,fill=black!100,thick,inner sep=0pt,minimum size=1mm}]
\node (v1) at ( 0, 0) [vertex, label=below:$1^*$] {};
\node (v2) at ( 1,1) [vertex, label=above:$3^*$] {};
\node (v3) at (2,0) [vertex, label=below:$2^*$] {};
\draw [very thick] [-] (v1) to (v2);
\draw  [-] (v3) to (v2);
\end{tikzpicture}
\]
Therefore $\KPomega = -\psi_3 + \psi_{12} - \psi_{12} + 2\psi_{111} = -\psi_3 + 2\psi_{111}$.

\end{example}

Observe that although the labeled poset in the previous example has a pointed $\Pomega$-partition with weight $(1, 2)$, the composition $(1, 2)$ is not in the $\psi$-support of $\KPomega$. 

\begin{quest}
What is the $\psi$-support of $\KPomega$?
\end{quest}

One should note that the combinatorial interpretation given in Theorem \ref{Combinatorial Description of Psi} is consistent with the description given in \cite{AlexanderssonSulzgruber} %ALEXANDERSSON AND SULZGRUBER
(or Corollary~\ref{Naturally Labeled Expansion} above)
when $\Pomega$ in naturally labeled.  %The only pointed $\Pomega$-partitions in the naturally labeled case are when each $P_i$ has a unique minimal element.

\subsection{Murnaghan-Nakayama rule}\label{Murnaghan-Nakayama rule section}
In this section, we will compare Theorem~\ref{Combinatorial Description of Psi} to the Murnaghan-Nakayama rule, which expresses a Schur symmetric function in terms of power sum symmetric functions.

We begin with some background on symmetric functions and tableaux. (For more information, see \cite{SaganBook, Stanley2}.)
A \emph{symmetric function} in the variables $x_1, x_2, \dots$ (with coefficients in $\C$) is a formal power series $f(\mathbf{x}) \in \C[[\mathbf{x}]]$ of bounded degree such that, for any composition $\alpha$, the coefficient of $x_1^{\alpha_1}x_2^{\alpha_2}\cdots x_k^{\alpha_k}$ equals the coefficient of $x_{i_1}^{\alpha_1}x_{i_2}^{\alpha_2}\cdots x_{i_k}^{\alpha_k}$ whenever $i_1, i_2, \dots, i_k$ are distinct. We denote the algebra of symmetric functions by $\Lambda$. % = \bigoplus_{n \geq 0} \Lambda_n$, graded by degree.
Clearly every symmetric function is also quasisymmetric.

%We will consider the following bases for $\Lambda$: the power sum symmetric functions and the Schur functions.  
The \emph{power sum symmetric function basis} $\{p_\lambda\}$, indexed by partitions $\lambda$, is given by $p_n = x_1^n +x_2^n + \cdots$ and $p_\lambda = p_{\lambda_1} p_{\lambda_2}\cdots$. In \cite{BallantineEtc}, it was shown that
\[\frac{p_{\lambda}}{z_\lambda} = \sum_{\alpha \sim \lambda}\psi_\alpha,\]
where the sum runs over all compositions $\alpha$ that rearrange to the partition $\lambda$.

For any partition $\lambda$, the \emph{Young diagram} of shape $\lambda$ is a collection of boxes arranged in left-justified rows such that row $i$ has $\lambda_i$ boxes.  If $\lambda$ and $\mu$ are partitions such that $\mu_i \leq \lambda_i$ for all $i$, then the Young diagram of the (skew) shape $\lambda / \mu$ is the set-theoretic difference between the Young diagram of shape $\lambda$ and the Young diagram of shape $\mu$.

%\begin{example}
%The following is a Young diagram of shape $652 / 21$.
%\[
%\ytableausetup
% {mathmode, boxsize=1em}
%\begin{ytableau}
%\none & \none & & & &\\
% \none &  & & & \\
% & \\
%\end{ytableau}
%\]
%\end{example}

A \emph{semistandard (Young) tableau} (SSYT) of (skew) shape $\lambda / \mu$ is a labeling of the boxes of the Young diagram of shape $\lambda / \mu$ such that the entries in the rows are weakly increasing from left to right, and the entries in the columns are strictly increasing from top to bottom.  If $T$ is an SSYT of shape $\lambda / \mu$, then we write $\lambda / \mu = \sh(T)$.  We say that $T$ has \emph{type} $\alpha = (\alpha_1, \alpha_2, \dots)$, denoted $\alpha = \type(T)$, if $T$ has $\alpha_i = \alpha_i(T)$ parts equal to $i$.  For any SSYT $T$ of type $\alpha$, we write $x^T = x_1^{\alpha_1(T)}x_2^{\alpha_2(T)}\cdots$. The \emph{skew Schur function} $s_{\lambda / \mu}$ is the formal power series
\[ s_{\lambda / \mu}(\x) =  \sum_Tx^T,\]
where the sum runs over all SSYT $T$ of shape $\lambda / \mu$. The Schur functions $\{s_\lambda\}$ for partitions $\lambda$ form a basis for $\Lambda$.

\begin{example}
The following is an SSYT of shape $652 / 21$ and type $(3, 3, 3, 1)$.
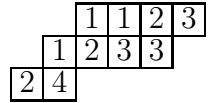
\[
\ytableausetup
 {mathmode, boxsize=1em}
\begin{ytableau}
\none & \none & 1 & 1 & 2 & 3 \\
 \none & 1 & 2 & 3 & 3 \\
 2& 4 \\
\end{ytableau}
\]
\end{example}

%SHOW THAT THIS IS THE (P, OMEGA)-PARTITION GENERATING FUNCTION FOR SOME POSET.
Define $P_{\lm}$ to be the poset whose elements are the squares $(i, j)$ of $\lm$, partially ordered componentwise.  Define a labeling $\omega_{\lm}\colon P_{\lm} \rightarrow [n]$ as follows:  first, the bottom square of the first column of $P_{\lm}$ is labeled $1$.  The labeling then proceeds in order up the first column, then up the second column, and so forth. %We will call $\omega_{\lm}$ the \emph{Schur labeling} of $P_{\lm}$.

\begin{example} The following is $(P_{\lm}, \omega_{\lm})$ when $\lambda = 652$ and $\mu = 21$.
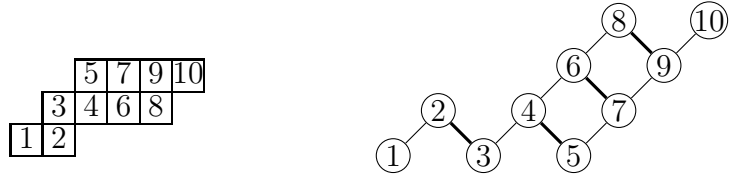
\[
\ytableausetup
 {mathmode, boxsize=1em, centertableaux}
 \vc{
\begin{ytableau}
\none & \none & 5 & 7 & 9 & 10 \\
 \none & 3 & 4 & 6 & 8 \\
 1& 2 \\
\end{ytableau}
}
\hspace{2 cm}
\vc{
\begin{tikzpicture}
 [auto,
 vertex/.style={circle,draw=black!100,fill=black!100,thick,inner sep=0pt,minimum size=1mm}, scale = .6]
%\node (v1) at ( 0, 0) [vertex, label=left:$1$] {};
%\node (v2) at ( 1,1) [vertex, label=left:$2$] {};
%\node (v3) at (2,0)  [vertex, label=left:$3$] {};
%\node (v4) at ( 3,1) [vertex, label=left:$4$] {};
%\node (v5) at (4, 0) [vertex, label=right:$5$] {};
%\node (v6) at (4, 2) [vertex, label=left:$6$] {};
%\node (v7) at (5, 1) [vertex, label=right:$7$] {};
%\node (v8) at (5, 3) [vertex, label=left:$8$] {};
%\node (v9) at (6, 2) [vertex, label=right:$9$] {};
%\node (v10) at (7, 3) [vertex, label=right:$10$] {};
\node (v1) [draw,circle,minimum size=.45cm,inner sep=0pt] at (0,0) {$1$};
\node (v2) [draw,circle,minimum size=.45cm,inner sep=0pt] at (1,1) {$2$};
\node (v3) [draw,circle,minimum size=.45cm,inner sep=0pt] at (2,0) {$3$};
\node (v4) [draw,circle,minimum size=.45cm,inner sep=0pt] at (3,1) {$4$};
\node (v5) [draw,circle,minimum size=.45cm,inner sep=0pt] at (4,0) {$5$};
\node (v6) [draw,circle,minimum size=.45cm,inner sep=0pt] at (4,2) {$6$};
\node (v7) [draw,circle,minimum size=.45cm,inner sep=0pt] at (5,1) {$7$};
\node (v8) [draw,circle,minimum size=.45cm,inner sep=0pt] at (5,3) {$8$};
\node (v9) [draw,circle,minimum size=.45cm,inner sep=0pt] at (6,2) {$9$};
\node (v10) [draw,circle,minimum size=.45cm,inner sep=0pt] at (7,3) {$10$};
\draw [-] (v1) to (v2);
\draw [very thick] [-] (v2) to (v3);
\draw [-] (v3) to (v4);
\draw  [very thick]  [-] (v4) to (v5);
\draw [-] (v4) to (v6);
\draw [-] (v5) to (v7);
\draw  [very thick]  [-] (v6) to (v7);
\draw [-] (v6) to (v8);
\draw [-] (v7) to (v9);
\draw  [very thick]  [-] (v8) to (v9);
\draw [-] (v9) to (v10);
\end{tikzpicture}
}
\]
\end{example}

It follows immediately from the definition of $\KPomega$ that 
\[ K_{(P_{\lm}, \omega_{\lm})}(\x) = \sum_Tx^T = s_{\lm}(\x), \]
where the sum runs over all SSYT $T$ of shape $\lm$.

We can express $s_{\lm}$ in terms of the power sum symmetric function basis using the following combinatorial rule, known as the \emph{Murnaghan-Nakayama rule}.

%\subsection{Symmetric Functions}
%A \emph{symmetric function} in the variables $x_1, x_2, \dots$ (with coefficients in $\C$) is a formal power series $f(\mathbf{x}) \in \C[[\mathbf{x}]]$ of bounded degree such that, for any composition $\alpha$, the coefficient of $x_1^{\alpha_1}x_2^{\alpha_2}\cdots x_k^{\alpha_k}$ equals the coefficient of $x_{i_1}^{\alpha_1}x_{i_2}^{\alpha_2}\cdots x_{i_k}^{\alpha_k}$ whenever $i_1, i_2, \dots, i_k$ are distinct.  We denote the algebra of symmetric functions by $\Lambda = \bigoplus_{n \geq 0} \Lambda_n$, graded by degree.
%
%We will consider the following bases for $\Lambda$: the power sum symmetric functions and the Schur functions.  The \emph{power sum symmetric function basis} $\{p_\lambda\}$, indexed by partitions $\lambda$, is given by $p_n = x_1^n +x_2^n + \cdots$, and $p_\lambda = p_{\lambda_1}p_{\lambda_2}\cdots$. 

%TAKEN FROM STANLEY 2%%%%%%%%%
A \emph{border strip} is a connected skew shape with no $2\times2$ square.  Define the \emph{height} $\height(B)$ of a border strip $B$ to be one less than its number of rows.

A \emph{border-strip tableau} of shape $\lambda / \mu$ and type $\alpha = (\alpha_1, \alpha_2, \dots)$ is an assignment of positive integers to the squares of $\lambda / \mu$ such that
\begin{itemize}
\item every row and column is weakly increasing,
\item the integer $i$ appears $\alpha_i$ times, and
\item the set of squares occupied by $i$ forms a border strip.
\end{itemize}
The \emph{height} of a border-strip tableau $T$, denoted $\height(T)$, is the sum of the heights of the border strips that make up $T$.
%%%%%%%%%%%%%%%%%%%%%%%%

%TAKEN FROM STANLEY 2%%%%%%%%%
%Let $\lambda$ and $\mu$ be partitions such that $\mu_i \leq \lambda_i$ for all $i$.  A \emph{semistandard (Young) tableau} (SSYT) of \emph{(skew) shape} $\lambda / \mu$ is an array $T = (T_{ij})$ of positive integers of shape $\lambda / \mu$ (i.e., $1 \leq i \leq l(\lambda), \mu_i < j \leq \lambda_i$) that is weakly increasing in every row and strictly increasing in every column.  If $T$ is an SSYT of shape $\lambda / \mu$, then we write $\lambda / \mu = \sh(T)$.  We say that $T$ has \emph{type} $\alpha = (\alpha_1, \alpha_2, \dots)$, denoted $\alpha = \type(T)$, if $T$ has $\alpha_i = \alpha_i(T)$ parts equal to $i$.  For any SSYT $T$ of type $\alpha$, we write $x^T = x_1^{\alpha_1(T)}x_2^{\alpha_2(T)}\cdots$.
%%%%%%%%%%%%%%%%%%%%%%%

%The \emph{Schur function basis} $\{s_\lambda\}$, indexed by partitions $\lambda$, is given
%\[s_\lambda = \sum_Tx^T,\]
%where the sum runs over all SSYTs $T$ of shape $\lambda$.

%\begin{defn}
%Let $\lambda / \mu$ be a skew shape.  The \emph{skew Schur function} $s_{\lambda / \mu}$ of \emph{shape} $\lambda / \mu$ is the formal power series
%\[ s_{\lambda / \mu} =  \sum_Tx^T,\]
%where the sum runs over all SSYT $T$ of shape $\lambda / \mu$.
%\end{defn}

%Therefore we have $s_{\lm} = K_{(P_{\lm}, \omega_{\lm})}(\x)$.  We will call $\omega_{\lm}$ the \emph{Schur labeling} of $P_{\lm}$.  From the definition of $s_{\lm}$ it may not be obvious that $s_{\lm}$ is in fact always a symmetric function.  

\begin{thm}[Murnaghan-Nakayama rule]\label{Murnaghan-Nakayama rule}
For partitions $\lambda$, $\mu$, and $\nu$,
\[ s_{\lambda / \mu} = \sum_{\nu}\chi^{\lambda / \mu}(\nu)\frac{p_\nu}{z_\nu}.\]
Here $\chi^{\lambda / \mu}(\nu) = \sum_T(-1)^{\height(T)}$, where the sum ranges over all border-strip tableaux of shape $\lambda / \mu$ and type $\nu$.  
\end{thm}

When $\mu = \varnothing$, Theorem \ref{Murnaghan-Nakayama rule} gives a change of basis formula for expressing Schur functions in terms of the power sum symmetric functions.  %For more on symmetric functions see, \cite{SaganBook, Stanley2}.

Since the quasisymmetric power sums refine the symmetric power sums, the Murnaghan-Nakayama rule also gives a description of the $\psi_\alpha$-expansion of $s_{\lambda/\mu}$. Since $\chi^{\lambda/\mu}(\nu)$ also does not depend on the order of the parts of $\nu$ \cite{Stanley2}, we find that the Murnaghan-Nakayama rule is equivalent to
\[ s_{\lambda / \mu} = \sum_{\alpha}\chi^{\lambda / \mu}(\alpha)\psi_\alpha,\]
where the sum ranges over all compositions $\alpha$.

In fact, this description agrees with the one obtained by Theorem~\ref{Combinatorial Description of Psi} above as we will next demonstrate.

\begin{prop}\label{MN Rule}
When $\Pomega = (P_{\lm}, \omega_{\lm})$, $\MinOne(\KPomega) = 0$ unless $\lm$ is a border strip, in which case $\MinOne(\KPomega) = (-1)^{\height(\lambda/\mu)}$.% where $T$ is a border strip tableau of shape $\lm$.  
%
%This is equivalent to the Murnaghan-Nakayama rule.  
\end{prop}
%\begin{prop}\label{MN Rule}
%When $\Pomega$ is a skew Schur function, meaning $\Pomega = (P_{\lm}, \omega_{\lm})$ for some $\lambda$ and $\mu$, the combinatorial description given in Theorem \ref{Combinatorial Description of Psi} is equivalent to the Murnaghan-Nakayama Rule.
%\end{prop}

\begin{proof}
If $\lm$ contains a $2 \times 2$ square, then $(P_{\lm}, \omega_{\lm})$ contains a chain $a \prec b \prec c$ with $\omega(a) < \omega(b) > \omega(c)$.  It follows from Lemma \ref{Avoids natural-strict} that $\MinOne(\KPomega) = 0$.
Similarly, if $(P_{\lm}, \omega_{\lm})$ is disconnected, then by Corollary~\ref{cor:connected}, we have $\MinOne(\KPomega) = 0$.
Therefore the only way to have $\MinOne(\KPomega) \neq 0$ is if $\lm$ is connected and contains no $2 \times 2$ square, that is, if it is a border strip.

Using the terminology of Lemma~\ref{Computing MinOne}, when $\lm$ is a border strip, the elements of $(P_{\lm}, \omega_{\lm})$ that lie in $I$ are the leftmost boxes in the rows of the Young diagram with shape $\lm$.  Therefore $|I|$ is equal to the number of rows of $\lm$.  The only element of $J$ is the box in the southwest corner. By Lemma \ref{Computing MinOne}, we have that $\MinOne(\KPomega) = (-1)^{\height(\lambda/\mu)}$.
%This is equivalent to the Murnaghan-Nakayama Rule since $\KPomega = s_{\lm}$.
\end{proof}

%\color{red}\textbf{}\color{black}

\begin{example}
The following is $(P_{\lm}, \omega_{\lm})$ when $\lambda = 6332$ and $\mu = 221$.
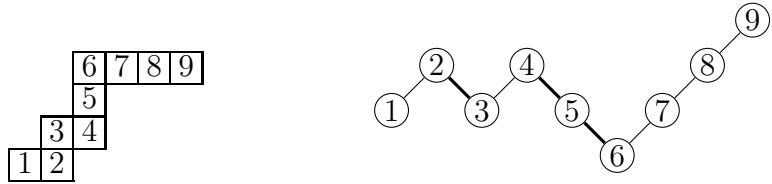
\[
 \vc{
\begin{ytableau}
\none & \none & 6 & 7 & 8 & 9 \\
\none & \none & 5 \\
 \none & 3 & 4 \\
 1& 2 \\
\end{ytableau}
}
\hspace{2 cm}
\vc{
\begin{tikzpicture}
 [auto,
 vertex/.style={circle,draw=black!100,fill=black!100,thick,inner sep=0pt,minimum size=1mm}, scale = .6]
%\node (v1) at ( 0, 0) [vertex, label=left:$1$] {};
%\node (v2) at ( 1,1) [vertex, label=left:$2$] {};
%\node (v3) at (2,0)  [vertex, label=left:$3$] {};
%\node (v4) at ( 3,1) [vertex, label=left:$4$] {};
%\node (v5) at (4, 0) [vertex, label=left:$5$] {};
%%\node (v6) at (4, 2) [vertex, label=left:$6$] {};
%\node (v7) at (5, -1) [vertex, label=left:$6$] {};
%%\node (v8) at (5, 3) [vertex, label=left:$8$] {};
%\node (v9) at (6, 0) [vertex, label=right:$7$] {};
%\node (v10) at (7, 1) [vertex, label=right:$8$] {};
%\node (v11) at (8, 2) [vertex, label=right:$9$] {};
\node (v1) [draw,circle,minimum size=.45cm,inner sep=0pt] at (0,0) {$1$};
\node (v2) [draw,circle,minimum size=.45cm,inner sep=0pt] at (1,1) {$2$};
\node (v3) [draw,circle,minimum size=.45cm,inner sep=0pt] at (2,0) {$3$};
\node (v4) [draw,circle,minimum size=.45cm,inner sep=0pt] at (3,1) {$4$};
\node (v5) [draw,circle,minimum size=.45cm,inner sep=0pt] at (4,0) {$5$};
\node (v7) [draw,circle,minimum size=.45cm,inner sep=0pt] at (5,-1) {$6$};
\node (v9) [draw,circle,minimum size=.45cm,inner sep=0pt] at (6,0) {$7$};
\node (v10) [draw,circle,minimum size=.45cm,inner sep=0pt] at (7,1) {$8$};
\node (v11) [draw,circle,minimum size=.45cm,inner sep=0pt] at (8,2) {$9$};
\draw [-] (v1) to (v2);
\draw [very thick] [-] (v2) to (v3);
\draw [-] (v3) to (v4);
\draw  [very thick]  [-] (v4) to (v5);
%\draw [-] (v4) to (v6);
\draw [very thick] [-] (v5) to (v7);
%\draw  [very thick]  [-] (v6) to (v7);
%\draw [-] (v6) to (v8);
\draw [-] (v7) to (v9);
%\draw  [very thick]  [-] (v8) to (v9);
\draw [-] (v9) to (v10);
\draw [-] (v10) to (v11);
\end{tikzpicture}
}
\]
Since $\lm$ is a border strip, it follows that $\MinOne(K_{(P_{\lm}, \omega_{\lm})}(\x)) = (-1)^3$.  We see that  $I = \{1, 3, 5, 6\}$ and $J = \{1\}$. %$\MinOne(K_{(P_{\lm}, \omega_{\lm})}(\x)) = (-1)^2$.
%\color{red} What is $I$ and what is $J$ \color{black}
\end{example}

\begin{cor}
	For partitions $\lambda$ and $\mu$ and compositions $\alpha$,
	\[ s_{\lambda / \mu} = \sum_{\alpha}\chi^{\lambda / \mu}(\alpha)\psi_\alpha.\]
	Here $\chi^{\lambda / \mu}(\alpha) = \sum_T(-1)^{\height(T)}$, where the sum ranges over all border-strip tableaux of shape $\lambda / \mu$ and type $\alpha$.  
\end{cor}
\begin{proof}
	By Proposition~\ref{MN Rule}, a pointed $(P_{\lambda/\mu}, \omega_{\lambda/\mu})$-partition $f^*$ is equivalent to a border-strip tableau $T$ of shape $\lambda/\mu$. In this case, $\wt(f^*) = \type(T)$ and $\sign(f^*) = (-1)^{\height(T)}$, so the result follows immediately from Theorem~\ref{Combinatorial Description of Psi}.
\end{proof}

\begin{example}
Let $\lambda = (6, 4, 4, 4, 2)$ and $\mu = (3, 2, 1, 1)$.  The following is a border strip tableau of shape $\lm$ and the pointed $(P_{\lm}, \omega_{\lm})$-partition $f^*$ that corresponds to it.
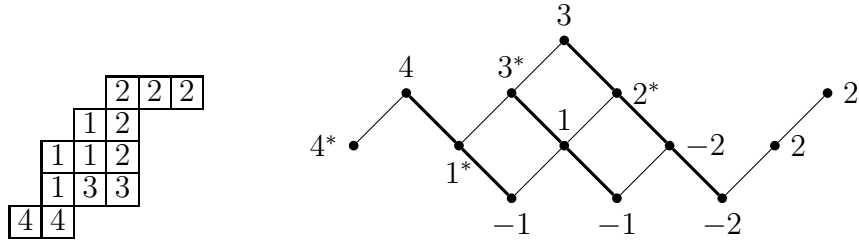
\[
 \vc{
\begin{ytableau}
\none & \none & \none & 2 & 2 & 2 \\
\none & \none & 1 & 2 \\
\none & 1 & 1& 2\\
 \none & 1 & 3 & 3 \\
 4& 4 \\
\end{ytableau}
}
\hspace{1 cm}
\vc{
\begin{tikzpicture}
 [auto,
 vertex/.style={circle,draw=black!100,fill=black!100,thick,inner sep=0pt,minimum size=1mm}, scale = .7]
\node (v1) at ( 0, 0) [vertex, label=left:$4^*$] {};
\node (v2) at ( 1, 1) [vertex, label=above:$4$] {};
\node (v3) at ( 2, 0) [vertex, label=below:$1^*$] {};
\node (v4) at (3, -1) [vertex, label=below:$-1$] {};
\node (v5) at (3, 1) [vertex, label=above:$3^*$] {};
\node (v6) at (4, 0) [vertex, label=above:$1$] {};
\node (v7) at (4, 2) [vertex, label=above:$3$] {};
\node (v8) at (5, -1) [vertex, label=below:$-1$] {};
\node (v9) at (5, 1) [vertex, label=right:$2^*$] {};
\node (v10) at (6, 0) [vertex, label=right:$-2$] {};
\node (v11) at (7, -1) [vertex, label=below:$-2$] {};
\node (v12) at (8, 0) [vertex, label=right:$2$] {};
\node (v13) at (9, 1) [vertex, label=right:$2$] {};
\draw [-] (v1) to (v2);
\draw [very thick] [-] (v2) to (v3);
\draw [very thick] [-] (v3) to (v4);
\draw [-] (v3) to (v5);
\draw [-] (v4) to (v6);
\draw [very thick] [-] (v5) to (v6);
\draw [-] (v5) to (v7);
\draw [very thick] [-] (v6) to (v8);
\draw [-] (v6) to (v9);
\draw [very thick] [-] (v7) to (v9);
\draw [-] (v8) to (v10);
\draw [very thick] [-] (v9) to (v10);
\draw [very thick] [-] (v10) to (v11);
\draw [-] (v11) to (v12);
\draw [-] (v12) to (v13);
\end{tikzpicture}
}
\]
\end{example}

\section{Acknowledgments}

The authors would like to thank Peter McNamara and Robin Sulzgruber for useful conversations.

%\printbibliography
\bibliographystyle{plain}
\bibliography{sample}

\end{document}